\documentclass[12pt, reqno]{amsart}
\usepackage{amsaddr}

\usepackage[utf8]{inputenc}
\usepackage{amsfonts}
\usepackage{latexsym}
\usepackage{amssymb}
\usepackage{amsmath}
\usepackage{mathtools}
\usepackage{esint}
\usepackage{xparse}

\usepackage[dvipsnames]{xcolor}

\usepackage{mathpazo}
\usepackage{domitian}
\usepackage[T1]{fontenc}

\usepackage[left=2.3cm, top=2.0cm,bottom=2.0cm,right=2.3cm]{geometry}

\newcommand{\sh}{\operatorname{sh}}
\newcommand{\bigsh}{\operatorname{SH}}

\newcommand{\orho}{\bar{\rho}}
\newcommand{\ow}{\bar{w}}
\newcommand{\pure}{\mathcal{S}}
\newcommand{\gen}{\mathcal{GS}}
\newcommand{\Sym}{\operatorname{Sym}}
\newcommand{\sgn}{\operatorname{sgn}}

\newcommand{\Conf}{\tau}
\newcommand{\sixv}{\mathcal{SV}}
\newcommand{\midedge}{\mathbb{M}}
\newcommand{\tor}{\mathbb{T}}
\newcommand{\indic}[1]{\mathrm{1}_{\{#1\}}}

\newcommand{\tr}{\operatorname{tr}}

\usepackage{enumerate}
\usepackage{mathrsfs}
\usepackage{amsthm}
\usepackage{adjustbox}
\usepackage{comment} 

\usepackage{hyperref} 
\hypersetup{
    colorlinks,
    linkcolor=OliveGreen,
    citecolor=OliveGreen,
    urlcolor=OliveGreen
}

\usepackage{graphicx}
\graphicspath{{Figures/}}

\numberwithin{equation}{section}

\usepackage{accents}

\definecolor{darkred}{RGB}{190,20,20}

\usepackage{tikz}
\usetikzlibrary{decorations.markings, arrows.meta,calc,backgrounds}
\tikzset{midarrow1/.style={postaction={decorate},decoration={markings, mark=at position 0.67 with {\arrow{Stealth}}}}}
\tikzset{midarrow2/.style={thick,postaction={decorate},decoration={markings, mark=at position 0.62 with {\arrow{Stealth}}}}}
\tikzset{endarrow/.style={thick,->,>={Stealth[length=2.4mm, width=1.8mm]}}}
\tikzset{startarrow/.style={thick,<-,>={Stealth[reversed, length=2.4mm, width=1.8mm]}}}
\pgfdeclarelayer{background}
\pgfdeclarelayer{foreground}
\pgfsetlayers{background,main,foreground}

\newtheorem{thm}{Theorem}[section]

\newtheorem{cor}[thm]{Corollary}
\newtheorem{lemma}[thm]{Lemma}
\newtheorem{df}[thm]{Definition}
\newtheorem{proposition}[thm]{Proposition}

\newtheorem{rem}{Remark}

\usepackage{nomencl}
\makenomenclature

\title[Explicit Correlation Functions in the Free-Fermion Regime]{Explicit correlation functions for the six-vertex model in the free-fermion regime}
\author{Samuel G. G. Johnston and Rohan Shiatis} 

\address{Department of Mathematics, King's College London, Strand, London, United Kingdom, WC2R 2LS}

\email{samuel.g.johnston@kcl.ac.uk}
\email{rohan.shiatis@kcl.ac.uk}

\subjclass{Primary: 60K35, 82B23, 82B20, 05C70}
\keywords{Six-vertex model, Kasteleyn matrix, Determinantal processes, Non-colliding processes}

\begin{document}

\begin{abstract}

In this article, we show that, in the free-fermion regime of the six-vertex model, all $k$-point correlation functions of vertex types admit a determinantal representation:
\begin{align*}
\mathbb{P}\Bigg( \bigcap_{p=1}^k \{ \text{vertex at } v^p \text{ has type } t_p \} \Bigg)
= \left( \prod_{p=1}^k a_{t_p} \right)
   \det\big[ L(x^i,y^j) \big]_{i,j=1}^{2k},
\end{align*}
where $t_1,\ldots,t_k \in \{1,\ldots,6\}$ label the six possible vertex types, and
$\{a_t : 1 \leq t \leq 6\}$ are the corresponding six-vertex weights. For each
$1 \leq p \leq k$, the four points
$x^{2p-1}, x^{2p}, y^{2p-1}, y^{2p} \in (\mathbb{Z}/2)^2$ are $t_p$-dependent choices among the
midpoints of the edges incident to $v^p$. The correlation kernel $L$ has the contour integral representation
\begin{align*}
L(x,y)
= \oint_{|w_1|=1} \oint_{|w_2|=1}
   \frac{dw_1}{2\pi i\, w_1}\,
   \frac{dw_2}{2\pi i\, w_2}\,
   w_1^{\,y_1 - x_1}\, w_2^{\,y_2 - x_2}\,
   h\big(c(x),c(y);w_1,w_2\big),
\end{align*}
where $h\big(c(x),c(y);w_1,w_2\big)$ is a simple rational function of $(w_1,w_2)$ that depends on $x$ and $y$ only through their orientations $c(x)$ and $c(y)$. 

As applications of our formula, we obtain a new expression for the frequency of vertices of each type and calculate the fine pairwise correlations of local vertex events in the free-fermion regime, showing that they decay quadratically with the distance between the vertices, in line with the quadratic decay known for related observables in the statistical-physics literature.

Our proof is self-contained: we construct a determinantal point process on $\mathbb{Z}^2$ and identify the six-vertex model as its pushforward under an explicit mapping.

\end{abstract}
\maketitle

\section{Introduction and Overview}

\subsection{The six-vertex model}

The six-vertex model (often abbreviated 6V) is a paradigmatic two-dimensional lattice model in statistical mechanics, originally introduced in the context of the residual entropy of ice in Pauling’s work \cite{Pauling1935}. In its modern formulation, it consists of configurations of edge orientations on a square lattice (or on a torus), subject to the so-called \emph{ice rule}. Namely, there must be two arrows in and two arrows out at each vertex, which yields six allowed local vertex types, from which the model takes its name \cite{Lieb1967}.

{
\renewcommand{\arraystretch}{1.5}
\setlength{\tabcolsep}{10pt}
\begin{figure}[!b]
\begin{tabular}{c c c c c c}
1&2&3&4&5&6\\[4pt]
\includegraphics{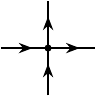} &
\includegraphics{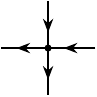} &
\includegraphics{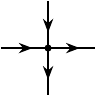} & 
\includegraphics{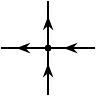} & 
\includegraphics{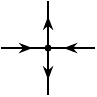} &
\includegraphics{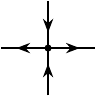} \\
$a_1$ & $a_2$ & $b_1$ & $b_2$ & $c_1$ & $c_2$ \\
\end{tabular}
\caption{All possible arrow patterns.}
\label{fig:arrowpatterns}
\end{figure}

}
\vspace{3pt}

We consider the six-vertex model on a graph with vertex set
either $V = \mathbb{Z}^2$ or the discrete torus $V = \mathbb{T}_n
= \{0,1,\ldots,n-1\}^2$. The edge set $E$ consists of pairs
$\{x,y\}$ of neighbouring vertices (with adjacency on
$\mathbb{T}_n$ taken modulo $n$ in each coordinate). Each vertex has degree four, i.e.\ for each $x \in V$ we have $\# \{ y \in V :  y \sim x \} = 4$,
where we write $y \sim x$ if $y$ and $x$ are neighbours. A \emph{six-vertex configuration} is an
orientation $\sigma:E \to V^2$ of the edges, i.e.\ a choice $\sigma(\{x,y\})\in\{(x,y),(y,x)\}$ for each $\{x,y\}\in E$,
subject to the ice rule that every vertex has in-degree and out-degree equal to two, so that for any fixed vertex $x \in V$,
\begin{align*}
\# \{ y \sim x: \sigma(\{x,y\}) = (x,y) \} = 2 \quad \text{and} \quad \# \{y \sim x : \sigma(\{x,y\}) = (y,x) \} = 2.
\end{align*}

We represent such a configuration by drawing an arrow from $x$ to $y$ on each edge $\{x,y\}$ when $\sigma(\{ x,y\} ) = (x,y)$. See Figure~\ref{fig:validconfig}, which illustrates
a six-vertex configuration on $\mathbb{T}_n$ in the case
$n = 4$.
The six possibilities for the orientation of the edges incident to a given vertex, which we associate with six vertex types in $\{1,\ldots,6\}$, are depicted in Figure~\ref{fig:arrowpatterns}. We also assign weights
\begin{equation*}
(a_1,a_2,b_1,b_2,c_1,c_2) \in (0,\infty)^6,
\end{equation*}
in correspondence with the order in which these vertices appear in Figure \ref{fig:arrowpatterns}.

\begin{figure}[!tb]
    \centering
    \includegraphics{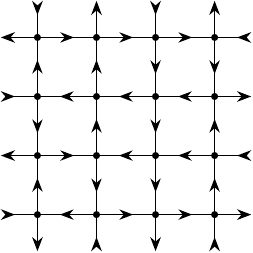}
    \caption{A six-vertex configuration on $\tor_4$.}
    \label{fig:validconfig}
\end{figure}

Working with the torus $\mathbb{T}_n$ for the time being, we assign to each six-vertex configuration $\sigma$ on $\tor_n$ a weight
\begin{equation}\label{eq:weightsixv}
    w(\sigma) = \prod_{(i,j) \in \mathbb{T}_n} w_{(i,j)}(\sigma),
\end{equation}
where $w_{(i,j)}(\sigma) \in \{a_1,a_2,b_1,b_2,c_1,c_2\}$ is the weight assigned to the local arrow pattern of $\sigma$ at a vertex $(i,j) \in \mathbb{T}_n$. We may 
thereby endow our configuration space with the Boltzmann measure defined by
\begin{equation}\label{eq:boltzmannsixv}
    \mathbb{P}_n(\sigma) = \frac{w(\sigma)}{Z_n}, \qquad Z_n = \sum_{\sigma' \in \sixv_n}w(\sigma'),
\end{equation}
where $\sixv_n$ denotes the set of six-vertex configurations on $\mathbb{T}_n$, that is, the set of all $\sigma$ on $\tor_n$ that obey the ice rule.
The probability measure $\mathbb{P}_n$, which depends on the weights $a_1,a_2,b_1,b_2,c_1,c_2$, then governs random six-vertex configurations on $\mathbb{T}_n$. Note that by construction, this probability measure is invariant under translations of space.

Given a six-vertex configuration $\sigma$ on $\mathbb{T}_n$, we can naturally associate it with a six-vertex configuration on $\mathbb{Z}^2$ by defining the configuration on the subset $\{0,\dots,n-1\} \times \{0,\dots,n-1\}$ and then extending it periodically. From this perspective, $\mathbb{P}_n$ may be regarded as a probability measure supported on $(n,n)$-periodic six-vertex configurations on the lattice $\mathbb{Z}^2$. With this in mind, we may now construct a probability measure $\mathbb{P}$ on six-vertex configurations on the lattice as the weak limit of $\mathbb{P}_n$ as $n \to \infty$. 
We are able to show the convergence of the finite-dimensional distributions of 
\begin{align*}
\mathbb{P}_n\left( \bigcap_{i=1}^k \{ \Conf_{v^i}(\sigma) = t_i \} \right) := \mathbb{P}_n\left( \bigcap_{i=1}^k \{ \text{Vertex $v^i$ has type $t_i$}\} \right)
\end{align*}
as $n \to \infty$, and conclude by Kolmogorov's extension theorem that there is a unique
 probability measure $\mathbb{P}$ on lattice six-vertex configurations such that 
\begin{align*}
\mathbb{P} \left( \bigcap_{i=1}^k \{ \Conf_{v^i}(\sigma) = t_i \} \right) = \lim_{n \to \infty}  \mathbb{P}_n\left( \bigcap_{i=1}^k \{ \Conf_{v^i}(\sigma) = t_i \} \right),
\end{align*}
where $v^i \in \mathbb{Z}^2$, $t_i \in \{1,\dots,6\}$ and $\Conf_v(\sigma)$ is the local vertex type at location $v$ in the six-vertex configuration $\sigma$. Thus, we are at liberty to speak of the probability law $\mathbb{P}$ governing six-vertex configurations on $\mathbb{Z}^2$ associated with a parameter regime $a_1$, $a_2$, $b_1$, $b_2$, $c_1$, $c_2$. 

Given a parameter regime, there are two fundamental questions to tackle:

\begin{itemize}
\item \textbf{Partition functions.} Can we calculate the partition function $Z_n$ for the six-vertex model on the torus and the associated free energy $
E(a_1,a_2,b_1,b_2,c_1,c_2) := \lim_{n \to \infty} \frac{1}{n^2} \log Z_n$? 
\item \textbf{Correlation functions.} Can we obtain an explicit understanding of the joint correlation functions $\mathbb{P} \left( \bigcap_{i=1}^k \{ \Conf_{v^i}(\sigma) = t_i \} \right)$? 
\end{itemize}

To describe the behaviour of six-vertex configurations at a qualitative level, it is useful to introduce the anisotropy parameter,
\begin{align} \label{eq:anis}
\triangle := \frac{a_1a_2  + b_1b_2 - c_1 c_2}{ 2 \sqrt{a_1a_2b_1b_2} }.
\end{align}
The possible values of the anisotropy parameter partition the model into three qualitatively distinct regimes. For $\triangle > 1$, the six--vertex model is in a \emph{ferroelectric} regime with frozen, ordered configurations; 
for $\lvert \triangle \rvert < 1$, it lies in the \emph{disordered (critical)} regime with algebraically decaying correlations; 
and for $\triangle < -1$, it exhibits \emph{antiferroelectric} order with a staggered, oscillatory structure.
We are particularly interested in the special subcase of the disordered regime 
\begin{align*}
\triangle = 0,
\end{align*}
which is known as the \emph{free-fermion} regime.

Shortly, we will state our main result, Theorem \ref{thm:main}, which provides a simple and explicit formula for the $k$-fold correlation functions $\mathbb{P}( \cap_{i=1}^k \{ \Conf_{v^i}(\sigma) = t_i \} )$ of the six-vertex model on $\mathbb{Z}^2$ in the free-fermion regime. Before this, however, let us review some related literature.

\subsection{Related work}

The foundational breakthrough in the exact solution of the six-vertex model was due to Lieb \cite{Lieb1967a,Lieb1967b}, who computed the free energy of the six-vertex model using a transfer-matrix method that involved the Bethe ansatz; see also Sutherland \cite{Sutherland1967}. A unified treatment of transfer-matrix methods can be found in Baxter's monograph \cite{baxter} on exactly solvable models.

In recent decades, researchers have focused on boundary condition effects and developed connections with combinatorics and random tilings. Reshetikhin and Sridhar \cite{Reshetikhin2017} and Aggarwal \cite{aggarwalshape} have discussed limit shapes. Borodin, Corwin, and Gorin \cite{BCG} study a version of the six-vertex model with stochastic weights on each line - see also \cite{BG2} - while Gorin and Nicoletti \cite{GN} study connections with random matrix distributions, proving that the Tracy-Widom distribution arises from the asymptotics of the Izergin-Korepin formula. Lis \cite{lis} studies the delocalisation of the six-vertex model. Duminil-Copin et al.\ \cite{DC} develop new proofs of existence and condensation of Bethe roots for the Bethe ansatz equation to give a short and explicit computation of the partition function in the regime $-1 < \triangle < 1$.  This list is by no means exhaustive. The reader might consult Zinn-Justin's article \cite{ZinnJustin2009} for an excellent review of connections with non-intersecting paths, domino tilings, and determinantal processes.

A distinct strand of work focuses on the free‐fermion regime, $\triangle = 0$ of the six‐vertex model and its equivalences with tiling and dimer models. Ferrari and Spohn \cite{FerrariSpohn2006} make explicit the correspondence (under domain‐wall boundary conditions) between the $\triangle=0$ six-vertex model and domino tilings of Aztec‐type regions; see also Section 7 of \cite{EKLP}. More recently, Aggarwal, Borodin, Petrov, and Wheeler \cite{ABPW} introduced a family of symmetric functions obtained as partition functions of six-vertex models, and used the algebraic Bethe ansatz to prove explicit formulas for them; see also Borodin and Petrov \cite{BP}. The resulting models were shown to form a determinantal process with explicit correlations in terms of contour integrals \cite{OR}. Duminil-Copin, Lis and Qian \cite{DCLQ} obtain the arrow-arrow correlations in the six-vertex model corresponding to the critical Ising model, using a dimer representation of the model obtained in Boutillier and de Tilière \cite{BdT} (see also \cite{DCLQ2, dubedat}). We mention in particular that the six-vertex model is in natural bijection with alternating sign matrices \cite{RobbinsRumseyASM,ZeilbergerASM,KuperbergASM}, and that its asymptotic and probabilistic structure has been the focus of substantial recent progress, see, e.g., \cite{AyyerChhitaJohansson2023}.

In recent years, substantial progress has been made in deriving explicit formulas for correlation functions in the six-vertex model. For instance, Bogoliubov, Pronko, and Zvonarev \cite{BPZ} showed that under domain-wall boundary conditions, the boundary one-point functions admit determinantal representations, with the authors obtaining closed‐form expressions in the free-fermion case. More generally, Colomo and Pronko \cite{ColomoPronko12} developed multiple‐integral representations for non‐local correlation functions, such as row‐configuration or emptiness‐formation probabilities, in the finite‐size six-vertex model with domain‐wall boundaries. More recently, Belov and Reshetikhin \cite{BelovReshetikhin20} derived analytic expressions for the two-point height‐function correlations at the free-fermion point and used them to distinguish between the different decay regimes: algebraic decay in the disordered phase and exponential decay in the antiferroelectric phase.

While these theories are closely related through various correspondences, they concern different probabilistic objects and rely on different algebraic frameworks. The present work instead studies correlation functions directly in the six-vertex model at the free-fermion point, specifically for periodic boundary conditions which are sparsely used in the literature, and obtains formulas intrinsic to the vertex model rather than through a reduction to dimers or Schur processes.

Further work in the mathematical physics literature on the correlation functions of six-vertex and XXZ-type models can be found in \cite{K, CJ}, and in particular, the series of papers \cite{grass1, grass2, grass3, grass4, grass5}.

\subsection{Our approach}
In the present article, we adopt a well-known bijection \cite{GN} between the six-vertex model and ensembles of paths on the lattice connecting mid-edges to mid-edges; see Figures~\ref{fig:arrowtosnake} and \ref{fig:validconfigtosnakes} for an illustration of how this bijection operates.

\begin{figure}[!b]
\setlength{\tabcolsep}{10pt}
\begin{tabular}{c c c c c c}\label{tab:arrowtosnake}
\includegraphics{a1.pdf} & 
\includegraphics{a2.pdf} & 
\includegraphics{b1.pdf} & 
\includegraphics{b2.pdf} & 
\includegraphics{c1.pdf} & 
\includegraphics{c2.pdf}  \\
\vspace{4pt}\\
\includegraphics{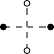} &
\includegraphics{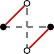} &
\includegraphics{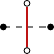} &
\includegraphics{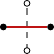} &
\includegraphics{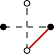} &
\includegraphics{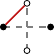} 

\end{tabular}
\caption{Arrow patterns and their corresponding path representations.}
\label{fig:arrowtosnake}
\end{figure}

 \begin{figure}[!tb]
    \centering
    \includegraphics{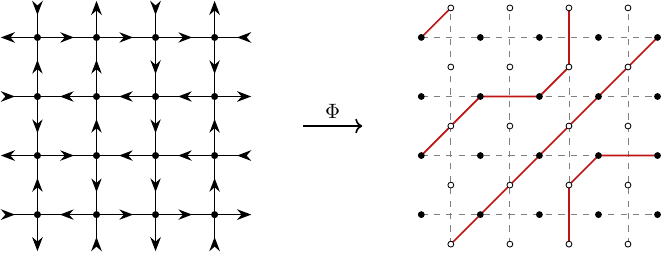}
    \caption{Applying the transformation to the configuration from Figure~\ref{fig:validconfig}.}
    \label{fig:validconfigtosnakes}
\end{figure}

We first restrict this bijection to vertices in the discrete $n$-by-$n$ torus $\mathbb{T}_n := \{0,1,\ldots,n-1\}^2$ with associated mid-edges $\midedge_n := \midedge_n^B \cup \midedge_n^W = \{ x + \tfrac{1}{2}e^1 : x \in \mathbb{T}_n \} \cup \{ x + \tfrac{1}{2}e^2 : x \in \mathbb{T}_n \}$, where the superscript $B$ denotes black (or horizontal) edges, and the superscript $W$ denotes white (or vertical) edges. Working on the torus allows us to avoid boundary effects and to consider a finite, translation-invariant version of the model with tractable formulas. 

Under the bijection in Figure~\ref{fig:arrowtosnake}, six-vertex configurations on $\mathbb{T}_n$ correspond to ensembles of non-intersecting paths connecting mid-edges in $\midedge_n$. A natural parameter regime for these path ensembles, assigning weights to horizontal, vertical and diagonal segments, pushes forward to the free-fermion weighting of the six-vertex model. Thus, in order to study random six-vertex configurations on the torus in the free-fermion regime, it is enough to study suitably weighted non-intersecting path ensembles on $\midedge_n$.

In this direction, let us define 
\begin{equation*}
    \mathcal{G}\mathcal{S}_n := \{ \orho: \midedge_n \to \midedge_n : \orho(x) \in \{ x, x+\mathbf{e}_1,x+\mathbf{e}_2,x+\frac{1}{2}(\mathbf{e}_1+\mathbf{e}_2) \} \}
\end{equation*}
to be the set of bijections of $\midedge_n$ sending each edge to either itself, the midedge to its east, north, or north east. We say a vertex $v$ of $\mathbb{T}_n$ is a crossing for $\orho \in \mathcal{G}\mathcal{S}_n$ if
\begin{align*}
\orho \left( v - \frac{1}{2}\mathbf{e}_1 \right) = v +  \frac{1}{2}\mathbf{e}_1 \quad \text{and} \quad \orho( v - \frac{1}{2}\mathbf{e}_2) = v +  \frac{1}{2}\mathbf{e}_2,
\end{align*}
in other words, the arrows connecting mid-edges `criss-cross' at $v$. 

We are interested in the subset
\begin{align*}
\mathcal{S}_n := \{ \rho \in \mathcal{G}\mathcal{S}_n : \text{$\rho$ has no crossings} \},
\end{align*}
since $\mathcal{S}_n$ may naturally be thought of as parameterising the set of nonintersecting path-ensembles on $\mathbb{T}_n$, or equivalently, six-vertex configurations on $\mathbb{T}_n$. In particular, if 
we define the partition function 
\begin{align} \label{eq:pf}
Z_n := \sum_{ \rho \in \mathcal{S}_n} A^{ \# \{ \text{rights} \}}   B^{ \# \{ \text{ups} \}}  C^{ \# \{ \text{diagonals} \} },
\end{align} 
then this partition function encodes the partition function of the six-vertex model on $\mathbb{T}_n$ in the free-fermion regime.

In the present article, we develop a  Kasteleyn theory for the six-vertex model on the torus (and ultimately its large-volume limit, $\mathbb{Z}^2$) that allows the straightforward expression of its partition and correlation functions in terms of determinants of translation-invariant operators. To outline this Kasteleyn approach, let us begin by considering an operator $K:\midedge_n \times \midedge_n \to \mathbb{C}$, which for midedges $x$ and $y$ in $\midedge_n$, takes the form
\begin{equation}\label{eq:Kdef0}
    K(x,y)=
    \begin{cases}
    1, &  \text{if $x=y$},\\[6pt]
    \alpha, &\text{if $y = x + e^1$ and $x$ black},\\[6pt]
    \beta, &\text{if $y = x + e^2$ and $x$ white,}\\[6pt]
    \gamma, &\text{if $y = x + (\tfrac{1}{2},\tfrac{1}{2})$},
    \end{cases}
\end{equation}
where the addition is to be taken mod $n$. Consider now that the determinant of $K$ may be written
\begin{align} \label{eq:pf2}
\det K = \sum_{ \orho \in \mathcal{G} \mathcal{S}_n } \mathrm{sgn}(\orho) \alpha^{ \# \{ \text{rights} \}}   \beta^{ \# \{ \text{ups} \}}  \gamma^{ \# \{ \text{diagonals} \} },
\end{align}
where the sum is taken over $\mathcal{G}\mathcal{S}_n$, the set of bijections on $\midedge_n$ that send each mid-edge either to itself or to a mid-edge to its north, east, or north-east. The reader will note the resemblance between the quantity $Z_n$ in \eqref{eq:pf} that we would like to calculate, and the determinant \eqref{eq:pf2} which we can indeed calculate using the tools of linear algebra. 

There are, however, two issues that must be addressed in relating the partition function $Z_n$ to determinants of the form $\det K$. The first is that the determinant gives a signed partition function, since each term is weighted by $\sgn(\orho)\in{-1,+1}$. Moreover, on the torus, these signs depend not only on the local crossings of the paths, but also on their winding around the two fundamental cycles. As in the standard toroidal Kasteleyn construction, we resolve this issue by introducing four twisted Kasteleyn operators $\{ K_\theta : \theta = (\theta_1,\theta_2) \in \{0,1\}^2 \}$.
A suitable linear combination of the four determinants $\det K_\theta$ then produces an unsigned partition function over $\mathcal{G}\mathcal{S}_n$.

\begin{figure}[!b]
\centering
\includegraphics{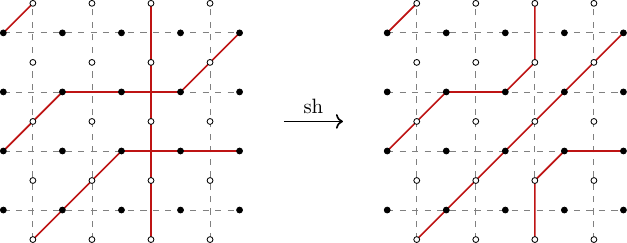}
\caption{The map $\sh$ from a generalised snake configuration to the corresponding uncrossed pure snake configuration.}
\label{fig:shapeoperator}
\end{figure}

The second issue is that the determinant is expanded over the larger set $\mathcal{G}\mathcal{S}_n$, rather than over the set $\mathcal{S}_n$ of non-intersecting path ensembles. To address this, we introduce a natural projection
\begin{align*}
\sh:\mathcal{G}\mathcal{S}_n\to\mathcal{S}_n
\end{align*}
which sends each $\orho\in\mathcal{G}\mathcal{S}_n$ to a configuration
\begin{align*}
\rho=\sh(\orho)\in\mathcal{S}_n
\end{align*}
by rewiring every crossing of $\orho$; see Figure~\ref{fig:shapeoperator}.
We then appeal to an idea introduced in a recent article of the authors \cite{JS}, where we investigated how Kasteleyn theory behaves under this shape map. Treating the determinants of the twisted Kasteleyn operators as signed partition functions gives as its Boltzmann distribution a signed measure on $\gen_n$ with determinantal correlations. More precisely, by studying cofactors of the operators $K_\theta$, we show that the signed weight of an event of the form $\{
\orho(x^1)=y^1,\ldots,\orho(x^k)=y^k
\}$ can be expressed as a linear combination of $k$-by-$k$ determinants whose entries are taken from the inverse Kasteleyn operators $K_\theta^{-1}$.

The crucial observation is that this signed measure pushes forward under the shape map $\sh$ to the probability measure on $\mathcal{S}_n$ induced by the partition function in \eqref{eq:pf}. Moreover, in contrast to the integrable snake model studied in \cite{JS}, the pushforward measure on $\mathcal{S}_n$ retains a determinantal correlation structure.

After carrying out this procedure, the translation invariance of the operators allows us to calculate the inverse Kasteleyn operators explicitly using Fourier analysis. It is then straightforward to send $n$ to infinity and obtain correlation formulas for the non-intersecting mid-edge path ensembles on $\mathbb{Z}^2$.

Finally, we pass these correlation formulas back through the bijection in Figure~\ref{fig:validconfigtosnakes}. Each vertex event in the six-vertex model corresponds to two mid-edge events in the associated path ensemble. Consequently, an event specifying the types of $k$ distinct vertices in the six-vertex model can be expressed as a $2k$-by-$2k$ determinant. This gives the formula stated in Theorem~\ref{thm:main}.

\subsection{Main result}
While there is a rich and expansive literature on the six-vertex model with a host of correlation formulas, to the best of our knowledge, there is no simple formula in the literature for the $k$-fold correlations 
\begin{align*}
\mathbb{P} \left( \bigcap_{i=1}^k \{ \Conf_{v^i}(\sigma) = t_i \} \right)  =: \mathbb{P}\left( \bigcap_{i=1}^k \{ \text{Vertex at $v^i$ has type $t_i$} \} \right) ,
\end{align*}
where $v^1,\ldots,v^k$ are elements of $\mathbb{Z}^2$ and $t_1,\ldots,t_k \in \{1,\ldots,6\}$ are labels for the six types of vertices. 
 The main result of the present article is a simple and explicit formula for such correlations in the free-fermion regime $\triangle = 0$. 

Before stating our main result, let us comment on the parameter space in the discrete setting. It will transpire that, in the free-fermion regime, we can relabel our parameter space in terms of the three variables $\alpha$, $\beta$, and $\gamma$, using the correspondence
\begin{align} 
    b_1 &\mapsto \beta, &b_2&\mapsto \alpha, \label{eq:reg1} \\
    c_1 &\mapsto \gamma, &c_2&\mapsto \gamma,\label{eq:reg2}\\
    a_1 &\mapsto 1, &a_2&\mapsto \gamma^2 - \alpha\beta. \label{eq:reg3}
\end{align}
We will justify these restrictions and analyse them in greater detail in Section~\ref{subsec:boltzmanncorr}. The limiting measure $\mathbb{P}$ will therefore also be parametrised by $\alpha$, $\beta$, and $\gamma$.

Let us introduce the set of mid-edges,
\begin{equation*}
    \midedge \coloneqq \left(\mathbb{Z}\times(\mathbb{Z}+\tfrac{1}{2})\right) \sqcup \left((\mathbb{Z}+\tfrac{1}{2})\times\mathbb{Z}\right).
\end{equation*} 
We can associate $\mathbb{M}$ naturally with the set of edges connecting vertices in the lattice.

Let $e^1 = (1,0)$, $e^2 = (0,1)$ represent unit vectors in $\tor_n$ or $\mathbb{Z}^2$. For each vertex $v = (v_1,v_2) \in \mathbb{Z}^2$, there are four mid-edges adjacent to $v$, these are given by 
\begin{align} \label{eq:cardinal}
W_v = v-\tfrac{1}{2} e^1, \quad S_v = v - \tfrac{1}{2}e^2,\quad  E_v = v+\tfrac{1}{2}e^1, \quad  N_v= v  +\tfrac{1}{2}e^2.
\end{align}
We refer to these edges as the edges west, south, east and north of $v$. Of course, a west edge of $v$ is an east edge of the lattice point $v - e^1$, etc.

\begin{figure}[h!tb]
\setlength{\tabcolsep}{10pt}
\begin{tabular}{c c c c c c}
1&2&3&4&5&6\\[6pt]
\includegraphics[width=1.28cm]{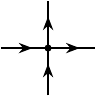} &
\includegraphics[width=1.28cm]{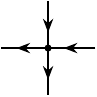} &
\includegraphics[width=1.28cm]{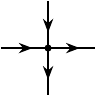} & 
\includegraphics[width=1.28cm]{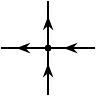} & 
\includegraphics[width=1.28cm]{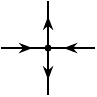} &
\includegraphics[width=1.28cm]{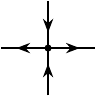} \\ [6pt]
$a_1$ & $a_2$ & $b_1$ & $b_2$ & $c_1$ & $c_2$ \\ [6pt]
\includegraphics{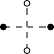} &
\includegraphics{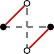} &
\includegraphics{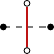} &
\includegraphics{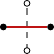} &
\includegraphics{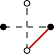} &
\includegraphics{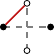} 
\end{tabular}
\caption{All possible arrow patterns.}
\label{fig:iceconfig}
\end{figure}

There is a natural way to associate the orientations of the four edges incident to a vertex $v$ with a collection of zero, one, or two bonds connecting the west/south mid-edges to the east/north mid-edges. Namely, if not every edge is oriented down or left (i.e., we are not in the second case of Figure \ref{fig:iceconfig}), connect the mid-edges that are oriented down or left with a bond. In the case that all four incident edges are oriented left or down (i.e., we are in the second case of Figure \ref{fig:iceconfig}), we pair the west edge with the north edge, and the south edge with the east edge. The six-vertex (ice) rule guarantees that these pairings are always possible.  See Figure~\ref{fig:iceconfig} for all six possible cases.

Given a collection $v^1,\ldots,v^k$ of points in $\mathbb{Z}^2$, and a collection of $t_1,\ldots,t_k$ of types in $\{1,\ldots,6\}$, we associate a collection of $4k$ mid-edges,
\begin{align*}
x^1,\ldots,x^{2k},y^1,\ldots,y^{2k}
\end{align*}
defined as follows. For each $1 \leq i \leq k$, $x^{2i-1},x^{2i},y^{2i-1},y^{2i}$ will be mid-edges incident to vertex $v^i$. Regardless of $t_i$, we will take 
\begin{align}\label{eq:biglist0}
x^{2i-1} = W_{v^i} \quad \text{and} \quad x^{2i} = S_{v^i}.
\end{align} 
But the choice of $y^{2i-1}$ and $y^{2i}$ will depend on both $v^i$ and $t_i$. The rule of thumb is that for $j = 2i-1$ or $j = 2i$, 
\begin{align*}
y^j := \text{point connected to $x^j$ by a bond as in Figure \ref{fig:iceconfig}},
\end{align*}
where if there is no such bond, we simply set $y^j = x^j$. More specifically, for each $1 \leq i \leq k$ we have 
\begin{equation}\label{eq:biglist}
\begin{aligned}
t_i = 1 &\implies y^{2i-1} = W_{v^i} \quad \text{and} \quad y^{2i} = S_{v^i},\\
t_i = 2 &\implies y^{2i-1} = N_{v^i} \quad \text{and} \quad y^{2i} = E_{v^i},\\
t_i = 3 &\implies y^{2i-1} = W_{v^i} \quad \text{and} \quad y^{2i} = N_{v^i},\\
t_i = 4 &\implies y^{2i-1} = E_{v^i} \quad \text{and} \quad y^{2i} = S_{v^i},\\
t_i = 5 &\implies y^{2i-1} = W_{v^i} \quad \text{and} \quad y^{2i} = E_{v^i},\\
t_i = 6 &\implies y^{2i-1} = N_{v^i} \quad \text{and} \quad y^{2i} = S_{v^i},
\end{aligned}
\end{equation}
where again, $W_v, S_v, E_v$ and $N_v$ are defined as in \eqref{eq:cardinal}.

We now define a correlation kernel $L:\midedge \times \midedge \to \mathbb{C}$ on the set of mid-edges as follows. We refer to mid-edges associated with horizontal edges as black edges, and mid-edges associated with vertical edges as white edges, and given $x \in \midedge$ write $c(x) = B$ or $c(x) = W$ for the colour of this edge. Note if $v \in \mathbb{Z}^2$ is a vertex, we have $c(W_v) = c(E_v) = B$ and $c(S_v) = c(N_v) = W$. Define the polynomial
\begin{align*}
\Delta(w_1,w_2)
= (1 + \alpha w_1)(1 + \beta w_2) - \gamma^2 w_1 w_2.
\end{align*}
For $c_1,c_2 \in \{B,W\}$ we now define the function $g(w_1,w_2,c_1,c_2)$ by the relations
\begin{align}
g(w_1,w_2,B,B) &= 1+ \beta w_2,\label{eq:gdef1}\\
g(w_1,w_2,W,W) &= 1 + \alpha w_1,\\
g(w_1,w_2,B,W) &= g(w_1,w_2,W,B) = - \gamma w_1^{1/2}w_2^{1/2},\label{eq:gdef3}
\end{align}
where we handle any ambiguity surrounding the complex square root in \eqref{eq:gdef3} in a moment. Now define
\begin{equation} \label{eq:Ldef0} 
        L(x,y) = \oint\!\oint_{|w_1|=|w_2|=1} \frac{dw_1}{2\pi i w_1}\,\frac{dw_2}{2\pi i w_2}\, \frac{w_1^{y_1 -x_1 }w_2^{y_2 - x_2}}{\Delta(w_1,w_2)} g(w_1,w_2,c(x),c(y))
    \end{equation}

Note that while both $w_1^{y_1 -x_1 }w_2^{y_2 - x_2}$ and $g(w_1,w_2,c(x),c(y))$ may have half-powers of $w_1$ and $w_2$, for all $x,y \in \mathbb{M}$ the quantity
\begin{align*}
w_1^{y_1 -x_1 }w_2^{y_2 - x_2}g(w_1,w_2,c(x),c(y))
\end{align*}
contains only integer powers of $w_1$ and $w_2$, thereby resolving the temporary ambiguity in the complex square root in \eqref{eq:gdef3}. Thus the integrand in \eqref{eq:Ldef0} is a rational function.

Note that $L$ depends on the parameters $\alpha,\beta,\gamma$ through $\Delta$ and $g$, both of which themselves depend on these parameters.

We now state our main result, which describes the correlations of the six-vertex model in the free fermion regime:

\begin{thm} \label{thm:main}
Let $\mathbb{P}$ be the probability measure governing the six-vertex model on $\mathbb{Z}^2$ in the free fermion regime with parameters $a_1,a_2,b_1,b_2,c_1,c_2$ indexed by $\alpha,\beta,\gamma$ as in \eqref{eq:reg1}, \eqref{eq:reg2}, \eqref{eq:reg3}. Let $v^1,\ldots,v^k$ be distinct points in $\mathbb{Z}^2$ and let $t_1,\ldots,t_k$ be labels in $\{1,\ldots,6\}$, as in Figure \ref{fig:iceconfig}. Then 
\begin{align*}
\mathbb{P}\left( \bigcap_{i=1}^k \{\text{Vertex $v^i$ has type $t_i$}\}\right)
= \prod_{i=1}^k a_{t_i} \det_{i,j=1}^{2k} L( x^i,y^j),
\end{align*}
where $L \colon \midedge \times \midedge \to \mathbb{C}$ is the correlation kernel on mid-edges $\midedge$ given in \eqref{eq:Ldef0}, and $x^1,\ldots,x^{2k}$, $y^1,\ldots,y^{2k}$ are mid-edges adjacent to the $v^i$ and defined in \eqref{eq:biglist0} and \eqref{eq:biglist}.
\end{thm}

To the authors' knowledge, Theorem~\ref{thm:main} gives the first explicit formula for arbitrary $k$-fold correlations of local vertex-type events in the six-vertex model. In particular, the fact that these correlations admit a representation as a $2k$-by-$2k$ determinant does not appear to have been observed previously.

We nevertheless emphasise that the six-vertex model in the free-fermion regime admits bijective realisations in terms of several integrable models, most notably domino tilings. It is therefore natural to expect that one could, at least in principle, recover our main result, Theorem~\ref{thm:main}, by passing to the domino-tiling representation and applying existing methods for deriving correlation formulas (for instance, the $k=1$ frequencies in Proposition~8.2 of \cite{CKP}, or the general $k\ge1$ formula in Theorem~4.2 of \cite{KOS} by taking a suitable limit). However, such an approach would require navigating a number of non-trivial intermediate steps, in large part due to the substantial bookkeeping to track the correspondence between observables. By contrast, our approach works directly within the six-vertex framework, yielding explicit correlation formulas in a self-contained manner.

\subsection{Vertex type frequencies}

The special case $k=1$ of our main result gives reasonably straightforward formulas for the frequency of each vertex type in the setting of Theorem \ref{thm:main}. 

Namely, for $\mathbf{w} = (w_1,w_2,\tilde{w}_1,\tilde{w}_2)$, define the vertex polynomials
\begin{align}
f_1(\mathbf w) \;=&\; (1+\beta w_2)(1+\alpha \tilde w_1) - \gamma^2 w_1 \tilde w_2, \label{eq:vertpoly1}\\[6pt]
f_2(\mathbf w) \;=&\; (\gamma^2 - \alpha\beta)\!\left(\gamma^2 \tilde w_1 w_2 - (1+\alpha \tilde w_1)(1+\beta w_2)\right) w_1 \tilde w_2,\\[6pt]
f_3(\mathbf w) \;=&\; \beta \tilde w_2\!\left((1+\beta w_2)(1+\alpha \tilde w_1) - \gamma^2 w_1 w_2\right),\\[6pt]
f_4(\mathbf w) \;=&\; \alpha w_1\!\left((1+\beta w_2)(1+\alpha \tilde w_1) - \gamma^2 \tilde w_1 \tilde w_2\right),\\[6pt]
f_5(\mathbf w) \;=&\; \gamma^2 \tilde w_2(1+\beta w_2)(w_1 - \tilde w_1),\\[6pt]
f_6(\mathbf w) \;=&\; \gamma^2 w_1(1+\alpha \tilde w_1)(\tilde w_2 - w_2). \label{eq:vertpoly6}
\end{align}

It is a simple but tedious calculation to verify that 
\begin{equation}\label{eq:summy}
\sum_{t=1}^6 f_t(\mathbf{w}) = \Delta(w_1,w_2)  \Delta(\tilde{w}_1,\tilde{w}_2). 
\end{equation}

We prove the following result as a corollary of Theorem \ref{thm:main}.

\begin{thm} \label{thm:frequency}
In the setting of Theorem \ref{thm:main}, for any $v \in \mathbb{Z}^2$ and $t \in \{1,\ldots,6\}$ we have
\begin{align*}
\mathbb{P}( \text{Vertex $v$ has type $t$} )
= \int_{(S^1)^4} 
\frac{ \mathrm{d}w_1}{2 \pi \iota w_1} \frac{ \mathrm{d}w_2}{2 \pi \iota w_2} \frac{ \mathrm{d}\tilde{w}_1}{2 \pi \iota \tilde{w}_1} \frac{ \mathrm{d}\tilde{w}_2}{2 \pi \iota \tilde{w}_2} \frac{ f_t(\mathbf{w})}{ \Delta(w_1,w_2) \Delta(\tilde{w}_1,\tilde{w}_2) }.
\end{align*}

\end{thm}
The reader will note that the fact that the polynomials satisfy \eqref{eq:summy} implies the (reassuring) fact that the probabilities $\mathbb{P}( \text{Vertex $v$ has type $t$} )$ sum over $1 \leq t \leq 6$ to unity.

Of course, the frequency of vertex types can be recovered from the partition function (or more specifically, the free energy). Indeed, to emphasise parameter dependence, write $Z_n(a_1,\ldots,a_6)$ for the partition function of the six-vertex model on the $n$-by-$n$ torus with parameters $a_1,\ldots,a_6$. If $\mathbb{P}_n$ is the associated probability measure, then for $1 \leq t \leq 6$ we have
\begin{align*}
\mathbb{E}_n[ \# \{ \text{Number of vertices of type $t$}\} ] = a_t\frac{\mathrm{d}}{\mathrm{d}a_t} \log Z_n(a_1,\ldots,a_6).
\end{align*}
Accordingly, since the torus contains $n^2$ vertices, by translation invariance, we have
\begin{align*}
\mathbb{P}_n( \text{Vertex $v$ has type $t$} ) = a_t\frac{\mathrm{d}}{\mathrm{d}a_t} \frac{1}{n^2} \log Z_n(a_1,\ldots,a_6).
\end{align*}
If we define the free energy by 
\begin{equation*}
E(a_1,\ldots,a_6) := \lim_{n \to \infty} \frac{1}{n^2} \log Z_n(a_1,\ldots,a_6),
\end{equation*} 
then, assuming we may interchange the limit and differentiation in $a_t$, we have 
\begin{equation*}
    \mathbb{P}( \text{Vertex $v$ has type $t$}) = a_t\frac{\mathrm{d}}{\mathrm{d}a_t} E(a_1,\ldots,a_6).
\end{equation*}

The free energy $E(a_1,\ldots,a_6)$ has been computed explicitly in various places \cite{Lieb1967a,Lieb1967b, Sutherland1967, DC}, but it is unclear to the authors how differentiating through this quantity relates to the frequency probabilities we obtain in Theorem \ref{thm:frequency}.

\subsection{Decay of correlations}
Consider the pairwise correlations between local vertex events. Recalling that $\{ \tau_{v^i}(\sigma) = t_i \}$ is the event that the vertex at $v^i \in \mathbb{Z}^2$ has type $t_i \in \{1,\ldots,6\}$, these pairwise correlations are given by
\begin{equation} \label{eq:corr}
\mathbb{P} ( \{ \tau_{v^1}(\sigma) = t_1 \} \cap \{ \tau_{v^2}(\sigma) = t_2 \} ) -\mathbb{P}(\{ \tau_{v^1}(\sigma) = t_1 \} ) \mathbb{P}(\{ \tau_{v^2}(\sigma) = t_2 \} ).
\end{equation}
The six-vertex model in the free-fermion regime is known to exhibit quadratic decay of pairwise correlations, meaning that the quantity in \eqref{eq:corr} is of order $O( \|v^2-v^1\|^{-2} )$, where $\|v^2-v^1\|$ denotes the Euclidean distance between the points $v^1$ and $v^2$ in $\mathbb{Z}^2$. This has been established in the statistical physics literature \cite{falco, Sutherland1967}. From a more probabilistic perspective, the six-vertex model in the free-fermion regime lies in the so-called \emph{liquid} phase of the associated dimer model \cite{KOS}, since the spectral curve $\Delta(z,w)=0$ meets the unit torus $|z|=|w|=1$ transversely. Related correlation calculations appear in, e.g., \cite{bounew} and \cite{GMT}. 

The case $k=2$ of our main result, Theorem \ref{thm:main}, provides a direct means of calculating these correlations explicitly. Although this calculation does not yield a fundamentally new phenomenon, we include it to demonstrate concretely that the correlation decay may be recovered directly and consistently from the formula in Theorem \ref{thm:main}. In particular, we establish in Theorem \ref{thm:corrasymp} that if $\mathbf{x} = (x,y)$ is a location far away from the origin, then
\begin{align*}
&\mathbf{P}\left( \text{The vertices at $(0,0)$ and $\mathbf{x}$ both have type $1$} \right) - \mathbf{P}( \text{Vertex $(0,0)$ has type $1$} )^2 \\ &= \frac{1}{2|\langle\mathbf{k},\mathbf{x}\rangle|^2} f(\mathbf{x}) + O(|\mathbf{x}|^{-3}),
\end{align*}
where $\mathbf{k}$ is a fixed complex vector depending on the parameters $\alpha,\beta,\gamma$, and $f(\mathbf{x})$ is an explicit bounded oscillatory function. Theorem \ref{thm:corrasymp} is proved using Theorem \ref{thm:main} and tools from \cite{KOS}.

Thus, although the quadratic order of decay is consistent with the behaviour known from the statistical-physics and dimer literature, our formula gives a direct and fully explicit description of the corresponding vertex-type correlations within the six-vertex model itself.

\subsection{{Article structure}}

The remainder of the article is structured as follows. 
In Section \ref{sec:bij}, we discuss a well-known bijection providing a correspondence between six-vertex configurations
on the torus and certain torus path ensembles. In Section \ref{sec:kast} we develop a Kasteleyn theory for the
six-vertex model, showing that the partition function can be written as a sum of determinants of certain operators
on the mid-edges, and that correlations of the model can be expressed in terms of the inverses of these operators.

The main novel ingredient of our approach appears in Section \ref{sec:scalinglim}, where we identify how the
correlations of the path ensemble interact with the bijection introduced in Section \ref{sec:bij}. This allows us to establish explicit correlation formulas for the six-vertex model on the torus. We then send the size of the torus to infinity, obtaining a proof of our main result, Theorem \ref{thm:main}.

\section{Correspondence between six-vertex model and snake configurations} \label{sec:bij}

The six-vertex model on the infinite lattice or on the discrete two-dimensional torus can be associated bijectively with certain collections of non-intersecting paths on the same space, which we call \textit{snake configurations}. We will see in particular that this bijection interacts particularly harmoniously with the parameter regime in the free-fermion case $\triangle = 0$. Indeed, in this case, the bijection sends Boltzmann measures to Boltzmann measures.

\subsection{The bijection}

In this section, we establish a bijection between six-vertex configurations on the discrete two-dimensional torus graph $\mathbb{T}_n = \mathbb{Z}_n^2$ and certain paths connecting the \emph{mid-edges} of points of $\mathbb{T}_n$. Namely, let us define the set of mid-edges as $\midedge_n := \midedge_n^B \cup \midedge_n^W$ where 
\begin{align*}
\midedge_n^B &:= \{ x + \tfrac{1}{2}e^1 : x \in \mathbb{T}_n \},\\
\midedge_n^W &:= \{ x + \tfrac{1}{2}e^2 : x \in \mathbb{T}_n \},
\end{align*}
with $e^1 = (1,0)$, $e^2 = (0,1)$, so that $\midedge_n^B$ are the midpoints of horizontal edges (which we call the black midpoints) and $\midedge_n^W$ are the midpoints of vertical edges (which we call the white midpoints).

Now we can specify the space of paths with which we work.

\begin{df} \label{df:snake}
A \emph{generalised snake configuration on $\midedge_n$} is a permutation $\orho:\midedge_n \to \midedge_n$ with the additional property that, for all $x \in \midedge_n$,
\begin{equation*}
    \orho(x) \in \begin{cases}
        \{x,x+ \tfrac{1}{2}(e^1 + e^2), x + e^1\} &\text{ if } x \text{ black,}\\
        \{x,x+ \tfrac{1}{2}(e^1 + e^2), x + e^2\} &\text{ if } x \text{ white.}
    \end{cases}
\end{equation*}
Let $\gen_n$ be the set of generalised snake configurations on $\midedge_n$.

A \emph{crossing} in a generalised snake configuration $\orho$ is a vertex $v \in \mathbb{T}_n$ for which
\begin{align*}
\orho(v - \tfrac{1}{2}e^1) = v+\tfrac{1}{2}e^1 \quad \text{and} \quad \orho(v - \tfrac{1}{2}e^2) = v+\tfrac{1}{2}e^2.
\end{align*}
A \emph{pure snake configuration} (which we will often denote by $\rho$) is a generalised snake configuration with no crossings. Denote $\pure_n$ to be the set of pure snake configurations on $\midedge_n$.
\end{df}

There is a bijective correspondence between six-vertex configurations on $\mathbb{T}_n$ and pure snake configurations on $\midedge_n$. This bijection is well-known and discussed at length (albeit in varying coordinate systems and perspectives) in, for instance, \cite{GN} and \cite{ZinnJustin2009}. 

Since the bijection is well-trodden, we will be content to give a somewhat informal statement and proof.
\begin{lemma} \label{lem:bij}
There is a bijection 
\begin{align*}
\Phi \colon \sixv_n \longleftrightarrow \pure_n
\end{align*}
between the set $\sixv_n$ of six-vertex configurations on $\mathbb{T}_n$ and the set $\pure_n$ of pure snake configurations on $\mathbb{T}_n$.
This bijection is constructed by looking at the local vertex type at $v$, and deciding the values of $\rho(v-\tfrac{1}{2}e^1)$ and $\rho(v-\tfrac{1}{2}e^2)$ according to the rule in Figure \ref{fig:arrowtosnake}:
\begin{itemize}
\item a red line from $v-\tfrac{1}{2}e^i$ to $v+\tfrac{1}{2}e^j$ corresponds to setting $\rho(v-\tfrac{1}{2}e^i) = v+\tfrac{1}{2}e^j$; and
\item no red line emerging from $v -\tfrac{1}{2}e^i$ corresponds to a fixed-point mid-edge, i.e., $\rho(v-\tfrac{1}{2}e^i) = v - \tfrac{1}{2}e^i$.
\end{itemize}
\end{lemma}

\begin{proof}

It suffices to check that, for all $\sigma \in \sixv_n$, $\Phi(\sigma)$ is a genuine pure snake configuration, and that the map $\Phi$ is invertible. Both details are not hard to check. For further details, see Gorin and Nicoletti \cite[Section 2]{GN}.

\end{proof}

We recall that Figure~\ref{fig:validconfigtosnakes} in the introduction offered an example of this bijection in action.

\subsection{The shape projector}\label{subsec:shape}

Recall the definitions of generalised snake configurations, crossings, and pure snake configurations given in Definition \ref{df:snake}.
There is a natural map 
\begin{align*}
\sh\colon\gen_n \twoheadrightarrow \pure_n
\end{align*}
that is defined by simply "straightening out" all crossings. More explicitly, given a  generalised snake configuration $\orho$, we define a pure snake configuration $\rho := \sh(\orho)$ by letting $\rho(v - \tfrac{1}{2}e^i) = \orho(v - \tfrac{1}{2}e^i)$, unless $v$ is a crossing, in which case we set 
\begin{align*}
\rho( v - \tfrac{1}{2}e^1) = v+\tfrac{1}{2}e^2 \quad \text{and} \quad \rho( v - \tfrac{1}{2}e^2) = v+\tfrac{1}{2}e^1.
\end{align*}
In other words, we `uncross' the crossing at $v$. Note that this map is certainly a projection, as $\pure_n \subset \gen_n$ and for all $\rho \in \pure_n$, we have that $\sh(\rho) = \rho$.
So we can remark the following.
\begin{rem}\label{rem:partGS}
    We may partition the set of generalised snakes as
    \begin{equation*}
        \gen_n = \bigsqcup_{\rho \in \pure_n}\sh^{-1}(\rho).
    \end{equation*}
\end{rem} 

Recall Figure \ref{fig:shapeoperator} from the introduction, which provided an illustration of the shape map in action.

\pgfmathsetmacro{\len}{0.8}
\pgfmathsetmacro{\size}{0.36}
\pgfmathsetmacro{\base}{-3pt}
\NewDocumentCommand{\crossing}{O{\size}}{
\begin{tikzpicture}[baseline=\base,scale = #1]
\begin{scope}
    \draw[gray, thick, dashed] (-\len/2,0) -- (0,0);
    \draw[gray, thick, dashed] (0,0) -- (\len/2,0);
    \draw[gray, thick, dashed] (0,-\len/2) -- (0,0);
    \draw[gray, thick, dashed] (0,0) -- (0,\len/2);
    \node[circle, fill=black, draw, inner sep=#1 pt] (left) at (-\len/2,0) {};
    \node[circle, fill=black, draw, inner sep=#1 pt] (right) at (\len/2,0) {};
    \node[circle, fill=white, draw, inner sep=#1 pt] (up) at (0,\len/2) {};
    \node[circle, fill=white, draw, inner sep=#1 pt] (down) at (0,-\len/2) {};

    \draw[thick, darkred] (left) -- (right);
    \draw[thick, darkred] (down) -- (up);
\end{scope}
\end{tikzpicture}
}
\NewDocumentCommand{\ainew}{O{\size}}{
\begin{tikzpicture}[baseline=\base,scale = #1]
\begin{scope}
    \draw[gray, thick, dashed] (-\len/2,0) -- (0,0);
    \draw[gray, thick, dashed] (0,0) -- (\len/2,0);
    \draw[gray, thick, dashed] (0,-\len/2) -- (0,0);
    \draw[gray, thick, dashed] (0,0) -- (0,\len/2);
    \node[circle, fill=black, draw, inner sep=#1 pt] (left) at (-\len/2,0) {};
    \node[circle, fill=black, draw, inner sep=#1 pt] (right) at (\len/2,0) {};
    \node[circle, fill=white, draw, inner sep=#1 pt] (up) at (0,\len/2) {};
    \node[circle, fill=white, draw, inner sep=#1 pt] (down) at (0,-\len/2) {};
\end{scope}
\end{tikzpicture}
}
\NewDocumentCommand{\aiinew}{O{\size}}{
\begin{tikzpicture}[baseline=\base,scale = #1]
\begin{scope}
    \draw[gray, thick, dashed] (-\len/2,0) -- (0,0);
    \draw[gray, thick, dashed] (0,0) -- (\len/2,0);
    \draw[gray, thick, dashed] (0,-\len/2) -- (0,0);
    \draw[gray, thick, dashed] (0,0) -- (0,\len/2);
    \node[circle, fill=black, draw, inner sep=#1 pt] (left) at (-\len/2,0) {};
    \node[circle, fill=black, draw, inner sep=#1 pt] (right) at (\len/2,0) {};
    \node[circle, fill=white, draw, inner sep=#1 pt] (up) at (0,\len/2) {};
    \node[circle, fill=white, draw, inner sep=#1 pt] (down) at (0,-\len/2) {};

    \draw[thick, darkred] (left) -- (up);
    \draw[thick, darkred] (down) -- (right);
\end{scope}
\end{tikzpicture}
}
\NewDocumentCommand{\binew}{O{\size}}{
\begin{tikzpicture}[baseline=\base,scale = #1]
\begin{scope}
    \draw[gray, thick, dashed] (-\len/2,0) -- (0,0);
    \draw[gray, thick, dashed] (0,0) -- (\len/2,0);
    \draw[gray, thick, dashed] (0,-\len/2) -- (0,0);
    \draw[gray, thick, dashed] (0,0) -- (0,\len/2);
    \node[circle, fill=black, draw, inner sep=#1 pt] (left) at (-\len/2,0) {};
    \node[circle, fill=black, draw, inner sep=#1 pt] (right) at (\len/2,0) {};
    \node[circle, fill=white, draw, inner sep=#1 pt] (up) at (0,\len/2) {};
    \node[circle, fill=white, draw, inner sep=#1 pt] (down) at (0,-\len/2) {};

    \draw[thick, darkred] (down) -- (up);
\end{scope}
\end{tikzpicture}
}
\NewDocumentCommand{\biinew}{O{\size}}{
\begin{tikzpicture}[baseline=\base,scale = #1]
\begin{scope}
    \draw[gray, thick, dashed] (-\len/2,0) -- (0,0);
    \draw[gray, thick, dashed] (0,0) -- (\len/2,0);
    \draw[gray, thick, dashed] (0,-\len/2) -- (0,0);
    \draw[gray, thick, dashed] (0,0) -- (0,\len/2);
    \node[circle, fill=black, draw, inner sep=#1 pt] (left) at (-\len/2,0) {};
    \node[circle, fill=black, draw, inner sep=#1 pt] (right) at (\len/2,0) {};
    \node[circle, fill=white, draw, inner sep=#1 pt] (up) at (0,\len/2) {};
    \node[circle, fill=white, draw, inner sep=#1 pt] (down) at (0,-\len/2) {};

    \draw[thick, darkred] (left) -- (right);
\end{scope}
\end{tikzpicture}
}
\NewDocumentCommand{\cinew}{O{\size}}{
\begin{tikzpicture}[baseline=\base,scale = #1]
\begin{scope}
    \draw[gray, thick, dashed] (-\len/2,0) -- (0,0);
    \draw[gray, thick, dashed] (0,0) -- (\len/2,0);
    \draw[gray, thick, dashed] (0,-\len/2) -- (0,0);
    \draw[gray, thick, dashed] (0,0) -- (0,\len/2);
    \node[circle, fill=black, draw, inner sep=#1 pt] (left) at (-\len/2,0) {};
    \node[circle, fill=black, draw, inner sep=#1 pt] (right) at (\len/2,0) {};
    \node[circle, fill=white, draw, inner sep=#1 pt] (up) at (0,\len/2) {};
    \node[circle, fill=white, draw, inner sep=#1 pt] (down) at (0,-\len/2) {};

    \draw[thick, darkred] (down) -- (right);
\end{scope}
\end{tikzpicture}
}
\NewDocumentCommand{\ciinew}{O{\size}}{
\begin{tikzpicture}[baseline=\base,scale = #1]
\begin{scope}
    \draw[gray, thick, dashed] (-\len/2,0) -- (0,0);
    \draw[gray, thick, dashed] (0,0) -- (\len/2,0);
    \draw[gray, thick, dashed] (0,-\len/2) -- (0,0);
    \draw[gray, thick, dashed] (0,0) -- (0,\len/2);
    \node[circle, fill=black, draw, inner sep=#1 pt] (left) at (-\len/2,0) {};
    \node[circle, fill=black, draw, inner sep=#1 pt] (right) at (\len/2,0) {};
    \node[circle, fill=white, draw, inner sep=#1 pt] (up) at (0,\len/2) {};
    \node[circle, fill=white, draw, inner sep=#1 pt] (down) at (0,-\len/2) {};

    \draw[thick, darkred] (left) -- (up);
\end{scope}
\end{tikzpicture}
}

Given a pure snake configuration $\rho$, we would like to describe the collection $\mathrm{sh}^{-1}(\rho)$ of generalised snake configurations $\bar{\rho}$ for which $\mathrm{sh}(\bar{\rho}) = \rho$. In this direction, given a pure snake configuration $\rho$ write 
\begin{equation*}
    \aiinew(\rho) = \{v \in \tor_n: \rho(v - \tfrac{1}{2}e^1) = v + \tfrac{1}{2}e^1 \text{ and } \rho(v - \tfrac{1}{2}e^2) = v + \tfrac{1}{2}e^2\}.
\end{equation*}
Then, it is clear that $\mathrm{sh}^{-1}(\rho)$ is encoded by choosing subsets $S$ of $\aiinew(\rho)$ as locations to replace paths by crossings. 
 \begin{rem}\label{rem:pure<->gen}
    For each $\rho \in \pure_n$, there is a bijection
    \begin{equation*}
    \{S \subset \aiinew(\rho)\}  \longleftrightarrow \mathrm{sh}^{-1}(\rho),
    \end{equation*} 
by associating with a subset $S$ of $\aiinew(\rho)$ the unique generalised snake configuration $\bar{\rho}$ in $\mathrm{sh}^{-1}(\rho)$ whose crossings occur at the vertices $S$. 
\end{rem}

\subsection{Boltzmann distributions and six-vertex bijections}\label{subsec:boltzmanncorr}
In this section, we discuss the parameterisation of the model and investigate the interaction between partition functions of six-vertex configurations and snake configurations through the bijective correspondence. We find that in the free-fermion regime, this map is particularly harmonious: Boltzmann distributions on six-vertex models correspond to (signed) Boltzmann distributions on general snake configurations.

To establish this connection, recall from the introduction that given parameters $a_1$, $a_2$, $b_1$, $b_2$, $c_1$, and $c_2$, we can define a probability measure $\mathbb{P}_n$ on six-vertex configurations on the torus as in Equations~\eqref{eq:weightsixv} and~\eqref{eq:boltzmannsixv}. We note that different parameter regimes may give rise to the same Boltzmann measure for the six-vertex model. For instance, for any $t > 0$, multiplying each of the parameters by $t$ leaves the probability measure invariant. A slightly more subtle observation is encoded in the following proposition. 

\begin{proposition} \label{prop:c1c2}
The Boltzmann probability measure $\mathbb{P}_n$ on six-vertex configurations on $\mathbb{T}_n$ associated with $(a_1,a_2,b_1,b_2,c_1,c_2)$ only depends on the parameters $c_1$ and $c_2$ through the product $c_1c_2$. 
\end{proposition}
\begin{proof}
Let $A_1$, $A_2$, $B_1$, $B_2$, $C_1$, and $C_2$ be the number of vertices of the six distinct types as displayed in Figure \ref{fig:iceconfig}. Note that the number of vertices for which there is a rightward arrow entering from the west must coincide with the number of vertices for which there is a rightward arrow exiting from the east. This equality reads $A_1 + B_1 + C_1 = A_1 + B_1 + C_2$, which reduces to $C_1 = C_2$. Since the weight of a configuration only depends on $c_1$ and $c_2$ by the expression $c_1^{C_1}c_2^{C_2} = (c_1c_2)^{C_1}$, the proposition follows.
\end{proof}

Note, of course, that the anisotropy parameter $\triangle$ (defined in Equation~\eqref{eq:anis}) is invariant under proportional scaling of all six parameters, and also only depends on $c_1,c_2$ through their product $c_1c_2$.
\vspace{3mm}

Since pure snake configurations are in bijection with six-vertex configurations, there is a natural Boltzmann measure that we can place on pure snake configurations. More interesting is to construct a natural Boltzmann measure on generalised snake configurations; rather, it turns out that the natural such measure is \textit{signed}. Through the bijection $\Phi$, one can see that the natural Boltzmann measure on pure snake configurations is weighted according to the value of pairs, $\rho(v - \tfrac{1}{2}e^1)$ and $\rho(v - \tfrac{1}{2}e^2)$ over $v \in \tor_n$. However, the utility of the (signed) Boltzmann measure we place on generalised snake configurations will turn out to be that it is defined only using the information of $\orho$ at individual vertices and the number of self-intersections. 

Let $e^3 = \tfrac{1}{2}(e^1 + e^2)$, and define
\begin{align} 
    A(\orho)&= \#\{x \in \midedge_n: \orho(x) = x + e^1\}, \label{eq:Adef} \\
    B(\orho)&= \# \{x \in \midedge_n: \orho(x) = x + e^2\},  \label{eq:Bdef} \\
    C(\orho)&= \# \{x \in \midedge_n: \orho(x) = x + e^3\}, \label{eq:Cdef}
\end{align}
which count the three types of non-fixed points of $\orho$. We also let 
 \begin{align*}
    S(\orho)&= \# \{ v \in \tor_n : \orho(v - \tfrac{1}{2}e^1) = v+\tfrac{1}{2}e^1 \quad \text{and} \quad \orho(v - \tfrac{1}{2}e^2) = v+\tfrac{1}{2}e^2\} 
 \end{align*}
count the number of crossings in a generalised snake configuration $\orho$. 

Given parameters $\alpha,\beta,\gamma > 0$, we define the weight of a generalised snake configuration by 
\begin{equation} \label{eq:genweight}
    \ow(\orho) = (-1)^{\#S(\orho)}\alpha^{\#A(\orho)} \beta^{\#B(\orho)} \gamma^{\#C(\orho)}.
\end{equation}
We note that the weight $\ow(\orho)$ is negative if $\orho$ has an odd number of crossings. 

We now define a (signed) partition function and (signed) Boltzmann measure on pure snake configurations by setting
\begin{equation*}
    P_n(\orho) = \frac{\ow(\orho)}{Z^{\gen}_n}, \qquad Z^{\gen}_n = \sum_{\orho' \in \gen_n}\ow(\orho'),
\end{equation*}
provided of course that $Z^{\gen}_n \neq 0$. We will see shortly that if $\gamma^2 - \alpha \beta >0$, then $Z^{\gen}_n > 0$. We emphasise that $P_n$ is not a probability measure but a \emph{signed probability measure}, i.e., a real-valued function on the set of generalised snake configurations with the property that $\sum_{\orho \in \gen_n} P_n (\orho) = 1$.

Recall now that we have the bijection $\Phi$ from Lemma~\ref{lem:bij} and the projection $\sh$ defined in Subsection~\ref{subsec:shape}. In particular, we can define a natural composition, 
\begin{equation*}
\bigsh \coloneqq \Phi^{-1} \circ \sh \quad \text{ so that } \quad \bigsh \colon \ \gen_n \  \twoheadrightarrow \  \sixv_n.
\end{equation*}

We can push forward the signed Boltzmann measures $P_n$ on generalised snake configurations to obtain probability measures on six-vertex configurations. More specifically, given a signed probability measure $P$ on a finite set $E$ (i.e. a function $P\colon E \to \mathbb{R}$ satisfying $\sum_{e \in E} P(e) = 1$), and a subset $F$ of $E$, write $P(F) \coloneqq \sum_{f \in F} P(f)$. Now let $\phi\colon E \twoheadrightarrow E'$ be a surjective function. Now we can define a \textbf{pushforward signed measure} $\phi_\# P$ on $E'$ by letting, for $e' \in E'$, 
\begin{equation*}
    (\phi_\# P)(e') := P(\phi^{-1}(e')).
\end{equation*}

Given a signed Boltzmann measure on generalised snake configurations $P_n$, we are interested in its pushforward under the shape map. In this direction we now state and prove the following result, justifying the change in parameters in our main result.

\begin{thm} \label{thm:pushforward}
Let $\alpha,\beta,\gamma > 0$ with $\gamma^2 - \alpha\beta > 0$. Then the pushforward $\mathrm{SH}_\# P_n$ is equal to $\mathbb{P}_n$, the Boltzmann measure on six-vertex configurations on $\mathbb{T}_n$, with parameters:
\begin{align} \label{eq:pushpara}
a_1 = 1, \quad a_2 = \gamma^2-\alpha \beta, \quad b_1 = \beta, \quad b_2 = \alpha, \quad c_1 = c_2 = \gamma.
\end{align}
Note that the collection of parameters $(a_1,a_2,b_1,b_2,c_1,c_2)$ in \eqref{eq:pushpara} lie in the free-fermion regime
\begin{align*}
\triangle := \frac{a_1a_2  + b_1b_2 - c_1 c_2}{ 2 \sqrt{a_1a_2b_1b_2} } = 0.
\end{align*}
Moreover, every Boltzmann measure on six-vertex configurations on $\mathbb{T}_n$ in the free-fermion regime with positive parameters arises this way. 
\end{thm}

The value of Theorem \ref{thm:pushforward} lies in the fact that while $P_n$ is not a probability measure in the usual sense, it has a pliant determinantal correlation structure which we describe below. We can study this determinantal structure under the $\mathrm{SH}$ map to obtain explicit information about the determinantal structure of the six-vertex model on $\mathbb{T}_n$.

\begin{proof}
Given parameters $\alpha,\beta,\gamma > 0$ with $\gamma^2 - \alpha\beta > 0$, let us define the \emph{pure weight} of a pure snake configuration by
\begin{align}\label{eq:weightsum}
w'(\rho) := \sum_{ \orho \in \mathrm{sh}^{-1}(\rho)} \ow (\orho).
\end{align}
Consider the following pictorial variables for the six possible vertex events expressed in terms of $\rho$.
\begin{align*}
    \ainew(\rho) &= \{ v \in \tor_n : \rho(v - \tfrac{1}{2}e^1) = v - \tfrac{1}{2}e^1 \text{ and } \rho(v - \tfrac{1}{2}e^2) = v - \tfrac{1}{2}e^2\}\\
    \aiinew(\rho) &= \{ v \in \tor_n : \rho(v - \tfrac{1}{2}e^1) = v + \tfrac{1}{2}e^2 \text{ and } \rho(v - \tfrac{1}{2}e^2) = v + \tfrac{1}{2}e^1\}\\
    \biinew(\rho) &= \{v \in \tor_n: \rho(v - \tfrac{1}{2}e^1) = v + \tfrac{1}{2}e^1\},\\
    \binew(\rho) &= \{v \in \tor_n: \rho(v - \tfrac{1}{2}e^2) = v + \tfrac{1}{2}e^2\},\\
    \cinew(\rho) &= \{v \in \tor_n: \rho(v - \tfrac{1}{2}e^2) = v + \tfrac{1}{2}e^1 \text{ and } \rho(v - \tfrac{1}{2}e^1) = v - \tfrac{1}{2}e^1\},\\
    \ciinew(\rho) &= \{v \in \tor_n: \rho(v - \tfrac{1}{2}e^2) = v - \tfrac{1}{2}e^2 \text{ and } \rho(v - \tfrac{1}{2}e^1) = v + \tfrac{1}{2}e^2\}.
\end{align*}
We can evaluate $w'(\rho)$ explicitly in terms of these quantities. Indeed, by virtue of Remark \ref{rem:pure<->gen} we have
\begin{align*}
    w'(\rho) &= \sum_{S \subset \aiinew(\rho)} (-\alpha\beta)^{\# S} (\gamma^2)^{\#\aiinew(\rho) - \# S} \alpha^{\#\biinew(\rho)} \beta^{\#\binew(\rho)} \gamma^{\#\cinew(\rho)+\#\ciinew(\rho)},\\
    &= (\gamma^2 - \alpha \beta)^{\#\aiinew(\rho)} \alpha^{\#\biinew(\rho)}\beta^{\#\binew(\rho)}\gamma^{\#\cinew(\rho)+\#\ciinew(\rho)}.
\end{align*}
Using the correspondence 
\begin{align*}
\Phi^{-1}: \pure_n \longleftrightarrow \sixv_n,
\end{align*}
we see that $w'(\rho)$ is precisely the weight of the corresponding six-vertex configuration, associated with the parameters in \eqref{eq:pushpara}. That is,
\begin{equation*}
w'(\rho) = w(\Phi^{-1}(\rho)).
\end{equation*}

To see that every Boltzmann measure on $\sixv_n$ in the free-fermion regime parametrises this way, first recall that we can freely scale our parameters such that $a_1 = 1$ and denote $b_1 = \beta$ and $b_2 = \alpha$. Secondly, apply Proposition~\ref{prop:c1c2} so that we may set $c_1 = c_2$ both to some $\gamma$. Finally, $\triangle = 0$ forces $a_2 = \gamma^2 - \alpha \beta$.
\end{proof}

Let us highlight a simple corollary of this pushforward.

\begin{cor}\label{cor:samepartfn}
    For the six-vertex model with the free-fermion parametrisation~\eqref{eq:pushpara}, and the generalised snake model with parameters $\alpha$, $\beta$ and $\gamma$, their corresponding partition functions $Z_n^{\gen}$ and $Z_n$ are equal. 
\end{cor}

\begin{proof}
    By Remark~\ref{rem:partGS}, Equation~\eqref{eq:weightsum} and the bijection $\Phi$, we can write
    \begin{equation*} 
    Z^{\gen}_n \coloneqq \sum_{ \orho \in \gen_n } \ow(\orho) =  \sum_{ \rho \in \pure_n } \sum_{ \orho \in \sh^{-1}(\rho) } \ow(\orho) =  \sum_{ \rho \in \pure_n }  w'(\rho) = \sum_{ \sigma \in \sixv } w(\sigma) \eqqcolon Z_n.
    \end{equation*}
\end{proof}

This, of course, justifies our earlier claim that $Z_n^\gen$ is positive when $\gamma^2 - \alpha \beta > 0$, as $Z_n$ is certainly positive when its corresponding parameters are positive.

\section{Kasteleyn Theory of the Snake Model}\label{sec:kast}

Kasteleyn theory was introduced by P.~W.~Kasteleyn in the 1960s in his solution to the dimer-counting problem on planar graphs, using Pfaffian methods and what are now called Kasteleyn orientations \cite{kasteleyn1961, kasteleyn1967}. It has since played a central role in the analysis of planar dimer and tiling models, including later developments in integrable probability and limit-shape phenomena \cite{CKP}.

In the previous section, we proved that every Boltzmann measure on six-vertex configurations on $\tor_n$ with positive parameters in the free-fermion regime arises as a pushforward of a signed Boltzmann measure $P_n$ on generalised snake configurations on $\midedge_n$. In this section, we begin by applying techniques from Kasteleyn theory to evaluate the partition function $Z_n^{\gen}$ and the correlations of the signed probability measure $P_n$ explicitly.

\subsection{The Kasteleyn operators}

For each $\theta = (\theta_1,\theta_2) \in \{0,1\}^2$, define the operator $K_\theta \colon \midedge_n \times \midedge_n \to \mathbb{C}$ by
\begin{equation}\label{eq:Kdef}
    K_\theta(x,y)=
    \begin{cases}
    1, &  \text{if $x=y$},\\[6pt]
    \alpha \exp\left(\frac{\pi \iota \theta_1}{n}\right), &\text{if $y = x + e^1$ and $x$ black},\\[6pt]
    \beta\exp\left(\frac{\pi \iota \theta_2}{n}\right), &\text{if $y = x + e^2$ and $x$ white,}\\[6pt]
    \gamma \exp\left(\frac{\pi \iota}{2n}(\theta_1 + \theta_2)\right), &\text{if $y = x + (\tfrac{1}{2},\tfrac{1}{2})$}.
    \end{cases}
\end{equation}

Recall from Corollary~\ref{cor:samepartfn} that under the parameter relation \eqref{eq:pushpara}, the partition functions $Z_n$ and $Z_n^{\gen}$ are equal. Our foundational result states that this partition function can be written as a sum of four determinants.

\begin{thm}\label{thm:partfn}
The partition function $Z^{\gen}_n$ has the explicit formula
    \begin{equation*}
        Z^{\gen}_n = \sum_{\theta \in \{0,1\}^2 } C_\theta \det(K_\theta)
    \end{equation*}
    where $C_\theta = \frac{1}{2}(-1)^{(\theta_1 + n + 1)(\theta_2 + n + 1)}$, and $K_\theta$ is defined in Equation \eqref{eq:Kdef}.
\end{thm}

Recall that for $\orho \in \gen_n$, $\orho$ is a permutation on $\midedge_n$ and therefore has a parity. Recall further that from \eqref{eq:Adef}, \eqref{eq:Bdef} and \eqref{eq:Cdef}, the integers $A(\orho), B(\orho)$ and $C(\orho)$ count the number of east, north and north-east connections between mid-edges. The main step in the proof of Theorem \ref{thm:partfn} is the following lemma, which says that the sign of $\orho$ can be expressed as a sum involving $A(\orho),B(\orho)$, and $C(\orho)$.

\begin{lemma}[Kasteleyn Lemma]\label{lem:sgngensnake}
    For all $\orho \in \gen_n$,
    \begin{equation}\label{eq:kastlemma1}
        \sgn(\orho)(-1)^{\#S(\orho)} = \sum_{\theta \in \{0,1\}^2} C_\theta \exp\left(\frac{\pi \iota \theta_1}{n}\left(A(\orho) + \tfrac{1}{2}C(\orho)\right) + \frac{\pi \iota \theta_2}{n}\left(B(\orho) + \tfrac{1}{2}C(\orho)\right)\right),
    \end{equation}
    where $C_\theta = \frac{1}{2}(-1)^{(\theta_1 + n + 1)(\theta_2 + n + 1)}$.
\end{lemma}

The proof of this lemma follows \emph{mutatis mutandis} from a very similar result in a recent preprint by the authors \cite[Lemma 3.5]{JS}. 

\begin{proof}
First of all, we claim that if the result holds for all pure snake configurations, then it holds for all generalised snake configurations. To see this, it suffices to show that if we replace any local configuration of the form $\aiinew$ by a crossing $\crossing$, the result still holds. Indeed, say $\orho$ and $\orho'$ are generalised snake configurations such that $\orho'$ is the same as $\orho$ but with one configuration of the form $\aiinew$ replaced by a crossing $\crossing$. Then
    \begin{itemize}
        \item $\orho'$ is $\orho$ composed with a transposition, so $\sgn(\orho') = -\sgn(\orho)$,
        \item $\#S(\orho') = \#S(\orho) + 1$, so $(-1)^{\#S(\orho')} = -(-1)^{\#S(\orho)}$,
        \item $\#A(\orho') = \#A(\orho) + 1$ and $\#B(\orho') = \#B(\orho) + 1$, whilst $\#C(\orho') = \#C(\orho) - 2$.
    \end{itemize}
    Overall, this means that if Equation~\eqref{eq:kastlemma1} holds for $\orho$, then it holds for $\orho'$. So it suffices to prove the result for pure snake configurations, that is, we want to show that
    \begin{equation}\label{eq:kastlemma2}
        \sgn(\rho) = \sum_{\theta \in \{0,1\}^2} C_\theta \exp\left(\frac{\pi \iota \theta_1}{n}\left(A(\rho) + \tfrac{1}{2}C(\rho)\right) + \frac{\pi \iota \theta_2}{n}\left(B(\rho) + \tfrac{1}{2}C(\rho)\right)\right),
    \end{equation}
    for any $\rho \in \pure_n$.

    To complete the proof, we need to make use of a topological lemma; see Section 3.4 of \cite{JS} for further information on the details.
    \begin{lemma}\label{lem:snakewind}
        Let $\rho \in \pure_n$ have disjoint cycle decomposition $\rho = c_1 c_2\cdots c_r$. For a cycle $c$ on $\midedge_n$, define its \textit{horizontal} and \textit{vertical winding numbers} by, respectively,
        \begin{equation*}
            q_1(c) = \frac{A(c) + \tfrac{1}{2}C(c)}{n}, \qquad q_2(c) = \frac{B(c) + \tfrac{1}{2}C(c)}{n}.
        \end{equation*}
        Then there exist coprime non-negative integers $q_1,q_2$ such that $q_i(c_j) = q_i$ for all $i = 1,2$ and $1 \leq j \leq r$.
    \end{lemma}
    
    If $\rho$ has disjoint cycle decomposition $c_1c_2\cdots c_r$, then it has sign $(-1)^{\sum_{i=1}^r (|c_i| + 1)}$ where $|c_i|$ denotes the length of cycle $c_i$. In particular, $|c_i| = A(c_i) + B(c_i) + C(c_i) = n(q_1(c_i) + q_2(c_i))$, so using Lemma~\ref{lem:snakewind}, we can write
    \begin{equation*}
        \sgn(\rho) = (-1)^{rn(q_1 + q_2) + r},
    \end{equation*}
    where $q_1$ and $q_2$ are the common values of $q_1(c_i)$ and $q_2(c_i)$ over $i$, respectively.

    We can simplify the RHS of Equation~\eqref{eq:kastlemma2} in the same manner. Using the definition of $C_\theta$ and using the integrality of $q_1$ and $q_2$, we have that it can be rewritten as
    \begin{align*}
        &= \sum_{\theta \in \{0,1\}^2} \tfrac{1}{2}(-1)^{(\theta_1 + n + 1)(\theta_2 + n + 1)} (-1)^{rq_1\theta_1 + rq_2 \theta_2}\\
        &= \sum_{\theta_2 \in \{0,1\}}(-1)^{(n+1)(\theta_2 + n + 1) + rq_2\theta_2}\left(\sum_{\theta_1 \in \{0,1\}} \tfrac{1}{2}(-1)^{\theta_1(\theta_2 + n + 1 + rq_1)}\right) \\ 
        &= \sum_{\theta_2 \in \{0,1\}}(-1)^{(n+1)(\theta_2 + n + 1) + rq_2\theta_2}\indic{\theta_2 + n + 1 + rq_1 \equiv 0 \bmod{2}}\\
        &= (-1)^{rq_1(n+1) + rq_2(n + 1 + q_1)}\\
        &= (-1)^{rn(q_1 + q_2)+ rq_1 + rq_2 + rq_1q_2}.
    \end{align*}
    However, we can conclude by noting that $rq_1 + rq_2 + rq_1q_2 \equiv r \bmod 2$ always holds whenever at least one of $q_1$ and $q_2$ is odd, which is certainly the case when they are coprime, as given by Lemma~\ref{lem:snakewind}.
    
\end{proof}

Recall from \eqref{eq:genweight} that the weight of a generalised snake configuration $\orho$ is given by
\begin{align*}
    \ow(\orho) = (-1)^{\#S(\orho)}\alpha^{\#A(\orho)} \beta^{\#B(\orho)} \gamma^{\#C(\orho)}.
\end{align*}

Before proceeding with the proof of Theorem \ref{thm:partfn}, we note here that Lemma \ref{lem:sgngensnake} gives a useful expression for the weight, which we record as an intermediate step.

\begin{lemma}\label{lem:weightform}
We have 
\begin{equation*}
\ow(\orho) = \mathrm{sgn}(\bar{\rho}) \sum_{ \theta \in \{0,1\}^2 } C_\theta \prod_{ x \in \midedge_n} K_\theta(x,\bar{\rho}(x)).
\end{equation*}
\end{lemma}
\begin{proof}

Let 
    \begin{equation*}
        F_\theta(\orho) = \frac{\pi \iota \theta_1}{n}\left(A(\orho) + \tfrac{1}{2}C(\orho)\right) + \frac{\pi \iota \theta_2}{n}\left(B(\orho) + \tfrac{1}{2}C(\orho)\right)
    \end{equation*}
be the exponentiated factor appearing in Lemma \ref{lem:sgngensnake}. Then we can note that
\begin{equation*}
\prod_{ x \in \midedge_n} K_\theta(x,\bar{\rho}(x)) = \exp( F_\theta(\orho)) \prod_{x \in \midedge_n}K_{00}(x,\bar{\rho}(x)).
\end{equation*}
Hence, by Lemma~\ref{lem:sgngensnake},
\begin{align*}
    \sgn(\orho) \sum_{ \theta \in \{0,1\}^2 } C_\theta \prod_{ x \in \midedge_n} K_\theta(x,\bar{\rho}(x)) &= \sgn(\orho) \sum_{ \theta \in \{0,1\}^2 } C_\theta \exp(F_\theta(\orho))\prod_{ x \in \midedge_n} K_{(0,0)}(x,\bar{\rho}(x))\\
    &= (-1)^{\#S(\orho)} \prod_{ x \in \midedge_n} K_{(0,0)}(x,\bar{\rho}(x)).
\end{align*}
The result follows by noting that 
\begin{equation*}
    \prod_{x \in \midedge_n} K_{00}(x,\bar{\rho}(x)) = \alpha^{\#A(\orho)} \beta^{\#B(\orho)} \gamma^{\#C(\orho)}. \qedhere
\end{equation*}  
\end{proof} 

We are now equipped to prove Theorem \ref{thm:partfn}.

\begin{proof}[Proof of Theorem~\ref{thm:partfn}]
    For a set $M$, let $\Sym(M)$ be the set of permutations on $M$. Recall that for a general operator, $K \colon M \times M \to \mathbb{C}$, the definition of the determinant is that
    \begin{equation*}
        \det(K) = \sum_{\rho \in \Sym(M)} \sgn(\rho) \prod_{x \in M} K(x,\rho(x)).
    \end{equation*}
    Hence, we have that
    \begin{equation*}
        \sum_{\theta \in \{0,1\}^2} C_\theta \det (K_\theta) = \sum_{\theta \in \{0,1\}^2} C_\theta \sum_{\orho \in \Sym(\midedge_n)} \sgn(\orho) \prod_{x \in \midedge_n} K_\theta(x,\orho(x)).
    \end{equation*}
    In particular, $\prod_{x \in \midedge_n}K_\theta(x,\orho(x))$ is supported on $\orho \in \gen_n \subset \Sym(\midedge_n)$, so the sum collapses to
    \begin{equation*}
        \sum_{\theta \in \{0,1\}^2} C_\theta \det (K_\theta) = \sum_{\orho \in \gen_n} \sgn(\orho) \sum_{\theta \in \{0,1\}^2} C_\theta \prod_{x \in \midedge_n} K_\theta(x,\orho(x)).
    \end{equation*}
    But now by Lemma~\ref{lem:weightform},
    \begin{align*}
        \sum_{\theta \in \{0,1\}^2} C_\theta \det (K_\theta) &=\sum_{\orho \in \gen_n} (-1)^{\#S(\orho)} \prod_{x \in \midedge_n} K_{(0,0)}(x,\orho(x)),\\
        &= \sum_{\orho \in \gen_n} \ow(\orho),
    \end{align*}
which is the definition of $Z_n^\gen$.
\end{proof}

\subsection{\texorpdfstring{Correlations under $P_n$}{Correlations under Pn}}

Recall that $P_n$ is a signed probability measure on generalised snake configurations. While $P_n$ is not a genuine probability measure, we can nonetheless speak of its correlations. Let $x^1,\dots,x^k$ and $y^1,\dots,y^k$ be mid-edges such that $y^i - x^i \in \{(0,0,e^1,e^2,e^3\}$. Then the correlations of $P_n$ are the quantities
\begin{align*}
      P_n\left[\bigcap_{i=1}^k \{\orho(x^i) = y^i\}\right]
      := \sum_{ \orho } P_n(\orho) \prod_{i=1}^k  \indic{\orho(x^i) = y^i}. 
\end{align*}
Of course, if $P_n$ were a genuine probability measure, this quantity would coincide with the expectation of the product of the random variable $\prod_{i=1}^k \indic{\orho(x^i) = y^i}$.

Our next result harnesses the determinantal structure of the partition function to express these correlations explicitly. 
This result is a variation on well-known results in the literature: it is very similar to Lemma 3.8 of \cite{JS}, but the ideas date back much earlier (e.g. \cite{kenyonaihp}). 

\begin{proposition} \label{prop:corr}
        Let $x^1,\dots,x^k \in \midedge_n$ and let $y^1,\ldots,y^k$ be suitable neighbouring mid-edges, so that $y^i - x^i \in \{(0,0), e^1, e^2, e^3\}$ for each $i$.  Then
        \begin{equation} \label{eq:corrn}
              P_n\left[\bigcap_{i=1}^k \{\orho(x^i) = y^i\} \right] =  \prod_{i = 1}^{k} K_{(0,0)}(x^i,y^i) \sum_{\theta \in \{0,1\}^2 } \frac{C_\theta \det(K_\theta)}{Z^{\gen}_n} q_\theta \det_{i,j = 1}^{k} K_\theta^{-1} (y^i,x^j),
        \end{equation}
        where $q_\theta = \exp\left(\frac{\pi \iota}{n}\big(\theta,\sum_{i=1}^k (y^i - x^i)\big) \right)$, with $\big(\theta,\sum_{i=1}^k (y^i - x^i)\big)$ being the standard inner product on $\mathbb{R}^2$.
\end{proposition}

\begin{proof}

We first note Jacobi's identity (see e.g. Horn and Johnson \cite[Section 0.8.4]{HJ}), which states that, given an invertible matrix $(K_{x,y})_{1 \leq x,y \leq n}$ and any distinct $x_1,\ldots,x_k$ and distinct $y_1,\ldots,y_k$ in $\{1,\ldots,n\}$, we have
\begin{equation*}
\sum_{ \sigma \in S_n } \prod_{i=1}^k \indic{ \sigma (x_i) = y_i } \mathrm{sgn}(\sigma) \prod_{x=1}^n K_{x,\sigma(x)} = \prod_{i=1}^k K_{x_i,y_i} \det_{x,y=1}^n(K_{x,y}) \det_{i,j = 1}^k ( K^{-1}_{y_i,x_j} ).
\end{equation*}
Now, from the definition of $P_n$,
\begin{equation*}
    P_n\left(\bigcap_{i=1}^k \{\orho(x^i) = y^i\}\right) = \sum_{\orho \in \gen_n} \frac{ \ow(\orho)\prod_{i=1}^k  \indic{\orho(x^i) = y^i}}{Z^\gen_n}.
\end{equation*}
From Lemma~\ref{lem:weightform}, and recalling that $\prod_{x \in \midedge_n}K_\theta(x,\orho(x))$ is only supported on $\orho \in \gen_n$,
\begin{equation*}
    P_n\left(\bigcap_{i=1}^k \{\orho(x^i) = y^i\}\right) = \sum_{\theta \in \{0,1\}^2}\frac{1}{Z_n^\gen}C_\theta\sum_{\orho \in \Sym(\midedge_n)} \prod_{i=1}^k \indic{\orho(x^i) = y^i} \sgn(\orho)\prod_{x \in \midedge_n}K_\theta(x,\orho(x)).
\end{equation*}
But now we can apply Jacobi's identity to find that
\begin{equation*}
      P_n\left(\bigcap_{i=1}^k \{\orho(x^i) = y^i\}\right)  = \sum_{\theta \in \{0,1\}^2 } \frac{C_\theta \det(K_\theta)}{Z^\gen_n} \prod_{i = 1}^{k} K_\theta(x^i,y^i)  \det_{i,j = 1}^{k} K_\theta^{-1} (y^i,x^j).
\end{equation*}
The result follows from noting that with $q_\theta$ as in the statement, we have 
\begin{equation*}
    \prod_{i = 1}^{k} K_\theta(x^i,y^i) = q_\theta \prod_{i = 1}^{k} K_{00}(x^i,y^i). \qedhere
\end{equation*}

\end{proof}
\subsection{Explicit form for the inverse operator} 

We saw at the end of the previous subsection that the correlations under $P_n$ can be described in terms of the inverse operator of $K_\theta^{-1}$. There are two obvious questions raised by this result. The first is whether $K_\theta^{-1}$ admits an explicit form, and the other is how this formula relates to events concerning the six-vertex model. The latter question takes a bit of setup and will be the subject of the following section, but the former question we can answer immediately.

As an ansatz based on comparisons with other exactly solvable lattice models, we might assume that $K_\theta^{-1}:\midedge_n \times \midedge_n \to \mathbb{C}$ is given by 
\begin{equation*}
K_\theta^{-1}(x,y) := \frac{1}{n^2} \sum_{\substack{w_1^{n}=1\\w_2^n=1}} w_1^{y_1-x_1} w_2^{y_2-x_2} f_\theta(w_1,w_2,c(x),c(y)),
\end{equation*}
where $f_\theta(w_1,w_2,c(x),c(y))$ depends on $\theta$, the roots $w_1,w_2$, but only depends on $x$ and $y$ through their colours $c(x)$ and $c(y)$. Explicitly, $c:\midedge_n \to \{B,W\}$ is the function,
\begin{equation*}
    c(x) = \begin{cases}
        B  &\text{ if } \quad x \in \midedge_n^B,\\
        W  &\text{ if } \quad x \in \midedge_n^W,
    \end{cases}
\end{equation*}
taking a mid-edge and returning its colour $b$ or $w$ (standing for black and white respectively). This indeed turns out to be a good guess.

Let us recall the definition of $K_\theta$ from Equation~\ref{eq:Kdef}. In particular, note that we can rewrite $K_\theta$ as 
\begin{equation}\label{eq:Kdef2}
   K_\theta(x,y)=
    \begin{cases}
   1, &  \text{if $x=y$},\\[6pt]
     \alpha_\theta, &\text{if $y = x + e^1$ and $c(x) = b$},\\[6pt]
     \beta_\theta &\text{if $y = x + e^2$ and $c(x) = w$,}\\[6pt]
     \gamma_\theta, &\text{if $y = x + e^3$}.
     \end{cases}
\end{equation}
where
\begin{align*}
 \alpha_\theta :=  \alpha \exp\left(\frac{\pi \iota \theta_1}{n}\right), \quad
 \beta_\theta :=  \beta\exp\left(\frac{\pi \iota \theta_2}{n}\right), \quad \text{and }
 \gamma_\theta :=  \gamma \exp\left(\frac{\pi \iota}{2n}(\theta_1 + \theta_2)\right).
 \end{align*}

Throughout this section, we will have square roots of complex numbers. Whilst we shall take the branch of the square root function $z \mapsto z^{1/2}$ for which $z^{1/2}$ always has argument taking values in $[0,\pi)$, it should never be relevant to the discussion, as we should find that the choice of branch cut should not matter.

For any $w_1,w_2 \in S^1 \coloneqq \{w \in \mathbb{C}: |w| = 1\}$, consider the $2\times2$ matrix $M_\theta(w_1,w_2)$ with rows and columns labelled by $B$ and $W$, given by 
\begin{equation*}
    M(w_1,w_2)=M_{\alpha,\beta,\gamma}(w_1,w_2) := 
\begin{pmatrix}
1+\alpha w_1 & \gamma w_1^{1/2} w_2^{1/2}\\[6pt]
\gamma w_1^{1/2} w_2^{1/2} & 1+\beta w_2
\end{pmatrix}.
\end{equation*}
The determinant of this matrix is given by 
\begin{equation*}
\Delta(w_1,w_2)
= (1 + \alpha w_1)(1 + \beta w_2) - \gamma^2 w_1 w_2.
\end{equation*}
Whenever $\Delta(w_1,w_2)\neq 0$, the inverse matrix is given by 
\begin{equation}\label{eq:defM}
M(w_1,w_2)^{-1} = \frac{1}{\Delta(w_1,w_2)}
\begin{pmatrix}
1+\beta w_2 & -\gamma w_1^{1/2} w_2^{1/2} \\[6pt]
-\gamma w_1^{1/2} w_2^{1/2} & 1+\alpha w_1
\end{pmatrix}.
\end{equation}

Now we can show that our ansatz is accurate, and explicitly write down the form of $f_\theta$ that is required.

\begin{proposition}\label{prop:Kinv}
    Let $x = (x_1,x_2)$ and $y = (y_1,y_2)$ be elements of $\midedge_n$. Then $K_\theta^{-1}: \midedge_n \times \midedge_n \to \mathbb{C}$ takes the explicit form
    \begin{equation} \label{eq:Kinv}
        K_\theta^{-1}(x,y) = \frac{1}{n^2} \sum_{\substack{w_1^n=1 \\ w_2^n=1}}w_1^{\,x_1-y_1} \, w_2^{\,x_2-y_2} \, \big[ M_{\alpha_\theta,\beta_\theta,\gamma_\theta}(w_1,w_2)^{-1}\big]_{c(x),\,c(y)}.
    \end{equation}
\end{proposition}
Let us emphasise that the expression in \eqref{eq:Kinv} has no half powers of $x$ and $y$. To see this, note that if $x$ and $y$ have the same colour, i.e.\ if $c(x) = c(y)$, then $x_1-y_1$ and $x_2 -y_2$ are both integers, and $M(w_1,w_2)^{-1}_{c(x),c(y)}$ contains no half powers, and if $x$ and $y$ have different colours, then $x_1-y_1$ and $x_2-y_2$ will both be non-integer half-integers, but the corresponding half powers of $w_1$ and $w_2$ will be coupled with the half powers in $M_\theta (w_1,w_2)^{-1}_{c(x),c(y)} = -\tfrac{1}{\Delta(w_1,w_2)}\gamma_\theta w_1^{1/2}w_2^{1/2}$.
\begin{proof}
We may assume without loss of generality that $\theta = (0,0)$, since both the definition \eqref{eq:Kdef2} for $K_\theta$ and \eqref{eq:Kinv} for $K_\theta^{-1}$ factor through $\alpha_\theta,\beta_\theta,\gamma_\theta$. Accordingly, we will suppress the subscript $\theta = (0,0)$ from now on in this proof.

With this in mind, we need only check that
\begin{equation}\label{eq:Jcheck}
    (K \circ J)(x,y) = \delta_{x,y},
\end{equation}
with $J$ defined as the RHS of Equation~\eqref{eq:Kinv} (when $\theta = (0,0)$). Along these lines, using the definition \eqref{eq:Kdef2}, we have that
\begin{align} \label{eq:convy}
(K \circ J)(x,y) &= \sum_{z \in \midedge_n} K(x,z) J(z,y) \nonumber \\
&= J(x,y) + \indic{c(x)=B} \alpha J( x+e^1,y) + \indic{c(x)=W}  \beta J( x+e^2,y) + \gamma J \left( x + e^3 , y \right) \nonumber \\
&= \frac{1}{n^2} \sum_{\substack{w_1^n=1\\w_2^n=1}} w_1^{\,x_1-y_1} \, w_2^{\,x_2-y_2} h(w_1,w_2)_{c(x),x(y)},
\end{align}
where
\begin{multline*}
    h(w_1,w_2)_{c(x),x(y)} = ( 1 + \indic{c(x)=B} \alpha w_1 + \indic{c(x)=W} \beta w_2 ) M^{-1}(w_1,w_2)_{c(x),c(y)} \\ + \gamma w_1^{1/2}w_2^{1/2} M^{-1}(w_1,w_2)_{\tilde{c}(x),c(y)},
\end{multline*}
with $\tilde{c}(x) \in \{B,W\}$ being the choice of colour such that $\tilde{c}(x) \neq c(x)$. The last term above follows from the fact that adding $e^3$ to a mid-edge changes its colour. But now we can rewrite $h$ as
\begin{align*}
    h(w_1,w_2)_{c(x),c(y)} &= \sum_{c \in \{B,W\}}\begin{bmatrix}
        1+\alpha w_1& \gamma w_1^{1/2}w_2^{1/2}\\
        \gamma w_1^{1/2}w_2^{1/2}&1 + \beta w_2
    \end{bmatrix}_{c(x),c}M^{-1}(w_1,w_2)_{c,c(y)}\\
    &= \sum_{c \in \{B,W\}}M(w_1,w_2)_{c(x),c}M^{-1}(w_1,w_2)_{c,c(y)}\\
    &= \indic{c(x) = c(y)}
\end{align*}

Plugging this into \eqref{eq:convy}, we find that
\begin{equation*}
    (K \circ J)(x,y) = \frac{1}{n^2}\sum_{\substack{w_1^n = 1\\w_2^n = 1}} w_1^{x_1 - y_1}w_2^{x_2 - y_2} \indic{c(x) = c(y)}
\end{equation*}
When $c(x) \neq c(y)$, this expression is $0$. And when $c(x) = c(y)$, $y_1 - x_1$ and $y_2 - x_2$ are both integers, and by discrete Fourier orthogonality, this expression is equal to $1$ if and only if $x_1 - y_1 = x_2 - y_2 = 0$. This confirms Equation~\eqref{eq:Jcheck} to be true.
\end{proof}

\section{The infinite volume limit}\label{sec:scalinglim}
Let us take stock of what we have seen so far. Recall that a generalised snake configuration is a permutation $\orho \colon \midedge_n \to \midedge_n$ on the mid-edges of $\tor_n$ with the property that each edge is either sent to itself, or to an edge immediately east, immediately north, or immediately north-east of a vertex. Given parameters $\alpha,\beta,\gamma$, we introduced a signed probability measure $P_n$ on generalised snake configurations on $\midedge_n$. We noted that the pushforward $\mathbb{P}_n = \bigsh_\#P_n$ of this probability measure, taking a generalised snake configuration and outputting a six-vertex configuration, is itself the Gibbs probability measure on six-vertex configurations with free-fermion weights given by the regime in \eqref{eq:pushpara}. 
We subsequently saw in Proposition \ref{prop:corr} that the correlations of the measure $P_n$ can be described explicitly in terms of determinants of the inverse operator, which is subsequently computed explicitly in Proposition \ref{prop:Kinv}. In this section, we take the infinite volume limit and pull back our results to the six-vertex model on $\mathbb{Z}^2$ to prove Theorem~\ref{thm:main}.

\subsection{The infinite volume limit of generalised snakes}

Precisely, we aim to describe the local correlation structure of the generalised snake model (and consequently, the six-vertex model) on 
\begin{equation*}
    \midedge \coloneqq \left(\mathbb{Z}\times(\mathbb{Z}+\tfrac{1}{2})\right) \sqcup \left((\mathbb{Z}+\tfrac{1}{2})\times\mathbb{Z}\right),
\end{equation*} 
the natural set limit $\midedge_n \uparrow \midedge$. Note that $\mathbb{M}$ is the set of mid-edges for $\mathbb{Z}\times  \mathbb{Z}$, on which we want to consider the six-vertex model. Denote by $\gen$ the set of generalised snakes on $\midedge$, equipped with the natural cylindrical $\sigma$-algebra, the $\sigma$-algebra generated by all the finite-coordinate projections.

Briefly, to be more precise about the mode of convergence, we may reconsider generalised snake configurations on $\midedge_n$ to be generalised snake configurations on $\midedge$ which are periodic, with period $n$ in the vertical and horizontal directions. In this sense, $P_n$ can be considered a measure on the set of generalised snake configurations on the lattice $\mathbb{Z}^2$, supported on these $n$-periodic generalised snakes. We want to show that a measure $P \coloneqq P^{\alpha,\beta,\gamma}$ on $\gen$ exists, such that $P_n$ converges weakly to $P$ as $n \to \infty$. Assuming that the measure $P$ is consistent when restricted to a measure on finite subsets, $P$ is uniquely determined by its local correlation structure by a version of Kolmogorov's extension theorem, since these correlations specify the signed probability over all cylinder events. However, the consistency is trivial to show, as the measures $P_n$ are consistent by their construction, and this property passes to $P$ through the limit.

Let us now write down the formula for the signed probabilities of the finite coordinate projections of generalised snakes. From now on, wherever unspecified, the integral $\oint$ refers to an integral over the anticlockwise circular contour of radius $1$ around the origin.

\begin{thm}\label{thm:localcorrGS}
    The measures $P_n$ converge weakly to a signed probability measure $P$ on the set of generalised snake configurations on the lattice, whose correlations are given explicitly by 
    \begin{equation*}
        P\left(\bigcap_{i=1}^k \{ \rho x^i = y^i \}\right) = \left(\prod_{i=1}^k K_{(0,0)}(x^i, y^i )\right)\det_{i,j = 1}^{k}L(x^{i},y^{j}),
    \end{equation*}
    where $x^1,\dots,x^k, y^1,\dots,y^k \in \midedge$ such that for each $i$, $y^i - x^i \in \{(0,0), e^1, e^2, e^3\}$, and
    \begin{equation} \label{eq:Ldef} 
        L(x,y) = \oint\!\oint
        \frac{dw_1}{2\pi i w_1}\,\frac{dw_2}{2\pi i w_2}\,w_1^{y_1 -x_1 }w_2^{y_2 - x_2}\big[M(w_1,w_2)^{-1}\big]_{c(x),\,c(y)}.
    \end{equation}
\end{thm}

\begin{proof}
    We show that $P_n$ converges on all finite-coordinate projections. By the discussion before the statement of the theorem, this defines the signed probability measure $P$.
    
    Examine the form of $P_n$ in Proposition \ref{prop:corr} in the limit as $n \to \infty$. By its independence of $n$, the product $\prod_{i=1}^k K_{(0,0)}(x^i,y^i)$ remains constant. For any choice of $\alpha, \beta$ and $\gamma$ in our parameter space, there will only be a finite set of pairs $(w_1,w_2)$ such that $\Delta(w_1,w_2) = 0$. Rephrasing, for co-finitely many $w_1$, $\Delta(w_1,w_2)$ will be non-zero for all $w_2 \in S^1$. So from Proposition~\ref{prop:Kinv}, by considering the double integral on the set of $w_1$ such that $\Delta(w_1,w_2) \neq 0$ for all $w_2 \in S^1$, we have the convergence of the Riemann sum,
    \begin{equation*}
    q_\theta \det_{i,j = 1}^{k} K_\theta^{-1} (y^i,x^j) 
    \to \det_{i,j = 1}^k\left(\oint\!\oint
    \frac{1}{2\pi \iota w_1} \frac{1}{2\pi \iota w_2} w_1^{y^i_1 - x_1^j} w_2^{y_2^i- x_2^j} \left[M(w_1,w_2)^{-1}\right]_{c(y^i),c(x^j)}\right).
    \end{equation*}
    Finally, recall from Theorem~\ref{thm:partfn} that $\sum_{\theta \in \{0,1\}^2}\frac{C_\theta\det(K_\theta)}{Z^\gen_n} = 1$. Therefore, the result follows by taking $n \to \infty$ in Equation~\eqref{eq:corrn} and noting that $M(w_1,w_2)^{-1}$ is a symmetric matrix.
\end{proof}

\subsection{Correlations of the six-vertex model}

We would like to consider the six-vertex model on $\tor \coloneqq \mathbb{Z}^2$ now, which we can make sense of via the shape map. Denote by $\sixv$ the set of six-vertex configurations on $\tor$ that satisfy the ice rule, and $\pure$ the set of pure snake configurations on $\midedge$. It is clear to see that, defined in the same way as in the finite case, there exists a bijection $\Phi \colon \sixv \longleftrightarrow \pure$ and a projection $\sh\colon \gen \twoheadrightarrow \pure$, which can be composed into a projection 
\begin{equation*}
    \bigsh = \Phi^{-1} \circ \sh \colon \gen \twoheadrightarrow \sixv.
\end{equation*}
For a choice of parameters $\alpha, \beta, \gamma$ such that $\gamma^2 - \alpha \beta > 0$, we will define the (possibly signed) probability measure $\mathbb{P} \coloneqq \mathbb{P}^{\alpha,\beta,\gamma}$ as the pushforward, 
\begin{equation}\label{eq:pushforwardP}
    \mathbb{P} \coloneqq \bigsh_{\#} P.
\end{equation}
In fact, the next theorem will show that $\mathbb{P}$ is a genuine probability measure.

To set up the following theorem, let us recall the local vertex types of the six-vertex model. For the reader's convenience, we have recalled them in Figure \ref{fig:iceconfig2}, listed in the order they appeared earlier in Figure~\ref{fig:iceconfig}. 

{
\renewcommand{\arraystretch}{1.5}
\setlength{\tabcolsep}{10pt}
\begin{figure}[!b]
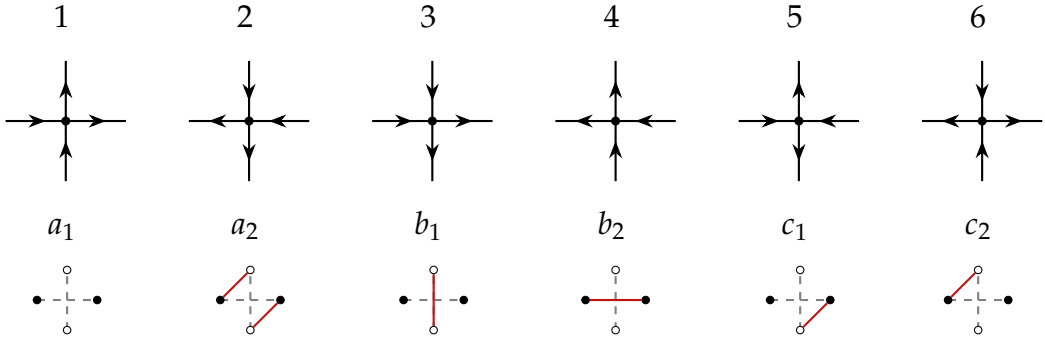

\begin{tabular}{c c c c c c}
1&2&3&4&5&6\\
\includegraphics{a1.pdf} &
\includegraphics{a2.pdf} &
\includegraphics{b1.pdf} & 
\includegraphics{b2.pdf} & 
\includegraphics{c1.pdf} &
\includegraphics{c2.pdf} \\
$a_1$ & $a_2$ & $b_1$ & $b_2$ & $c_1$ & $c_2$ \\
\includegraphics{a1new.pdf} &
\includegraphics{a2new.pdf} &
\includegraphics{b1new.pdf} &
\includegraphics{b2new.pdf} &
\includegraphics{c1new.pdf} &
\includegraphics{c2new.pdf} 
\end{tabular}
\caption{All possible arrow patterns.}
\label{fig:iceconfig2}
\end{figure}
}

Let $a_t \coloneqq a_t^{\alpha,\beta,\gamma}$ be the weight associated with a vertex of type $t$ under the parametrisation \eqref{eq:pushpara}. Let $\sigma$ be a random six-vertex configuration with law $\mathbb{P}$. For $v = (v_1,v_2) \in \mathbb{T}_n$, let $\tau_v(\sigma)$ be the local vertex type of $\sigma$ at $v$ and for $t = 1,\dots,6$, let $V_t(v) = \{\tau_v(\sigma) = t\}$ be the event that we have a vertex of type $t$ at $v$ in $\sigma$.  To relate these to generalised snake events, given a vertex $v \in \tor$, let us introduce the shorthand 
\begin{align*}
W_v = v-\tfrac{1}{2}e^1, \quad S_v = v - \tfrac{1}{2}e^2,\quad  E_v = v+\tfrac{1}{2}e^1, \quad  N_v = v +\tfrac{1}{2}e^2
\end{align*}
for the four mid-edges incident to $v$. Note that $W_v$ and $E_v$ are black mid-edges, and $S_v$ and $N_v$ are white mid-edges. With this notation, a generalised snake configuration is then simply a permutation $\orho:\midedge \to \midedge$ for which 
\begin{align*}
\orho (W_v) \in \{ W_v, N_v, E_v \} \quad \text{and} \quad \orho (S_v) \in \{ S_v, E_v, N_v \}
\end{align*}
for every $v \in \tor$. A pure snake configuration $\rho$ is a generalised snake configuration with no crossings, i.e., for which there exist no vertices $v \in \mathbb{T}$ at which $\rho (S_v) = N_v$ and $\rho (W_v) = E_v$.

Now, examining how events on the six-vertex model push forward to events on generalised snakes, for all local vertex types $t$ and $v \in \tor$, let 
\begin{equation}\label{eq:defU}
    U_t(v) = \bigsh^{-1}(V_t(v)).
\end{equation}
Then we have the relations
\begin{align*}
U_1(v)  &= \{ \orho(W_v) = W_v \} \cap \{ \orho (S_v) = S_v \},\\
U_2(v)  &=(\{ \orho (W_v) = N_v \} \cap \{ \orho (S_v) = E_v \} ) \cup ( \{ \orho (W_v) = E_v \} \cap \{ \orho (S_v) = N_v \} ),\\
U_3(v)  &= \{ \orho (W_v) = W_v \} \cap \{ \orho (S_v) = N_v \}, \\
U_4(v)  &= \{ \orho (W_v) = E_v \} \cap \{ \orho (S_v) = S_v \}, \\
U_5(v)  &= \{ \orho (W_v) = W_v \} \cap \{ \orho (S_v) = E_v \}, \\
U_6(v)  &= \{ \orho (W_v) = N_v \} \cap \{ \orho (S_v) = S_v \}.
\end{align*}
Note that there is almost a one-to-one relationship between vertex events for the six-vertex model and the intersections of two distinct events where a generalised snake configuration sends a particular mid-edge to a particular mid-edge. The one exception is in the second case, where there are two possibilities for how a vertex of type $2$ might arise: either with or without a crossing at $v$. 

Finally, our main result is a consequence of a nice cancellation that occurs, meaning correlations remain reasonably straightforward under the pullback.

\begin{thm} \label{thm:premain}
Let $v^1,\ldots,v^k$ be vertices in $\mathbb{Z}^2$. Then we have 
\begin{align*}
\mathbb{P} \left( \bigcap_{i=1}^k V_{t_i}(v^i) \right) = \left(\prod_{ j =1}^k a_{t_j}\right) \det_{i,j=1}^{2k}L(x^i,y^j),
\end{align*}
where $L$ is as in \eqref{eq:Ldef}, and $\{ x^i,y^j : 1 \leq i,j \leq 2k \}$ are mid-edges in $\midedge$ defined from the intersection $ \bigcap_{i=1}^k V_{t_i}(v^i)$ as follows. For each $1 \leq i \leq k$ we have
\begin{align*}
x^{2i-1} = W_{v^i} \quad \text{and} \quad x^{2i} = S_{v^i}.
\end{align*} 
As for $y^1,\ldots,y^{2k}$, for each $1 \leq i \leq k$ we have
\begin{align*}
t_i = 1 &\implies y^{2i-1} = W_{v^i} \quad \text{and} \quad y^{2i} = S_{v^i}\\
t_i = 2 &\implies y^{2i-1} = N_{v^i} \quad \text{and} \quad y^{2i} = E_{v^i}\\
t_i = 3 &\implies y^{2i-1} = W_{v^i} \quad \text{and} \quad y^{2i} = N_{v^i}\\
t_i = 4 &\implies y^{2i-1} = E_{v^i} \quad \text{and} \quad y^{2i} = S_{v^i}\\
t_i = 5 &\implies y^{2i-1} = W_{v^i} \quad \text{and} \quad y^{2i} = E_{v^i}\\
t_i = 6 &\implies y^{2i-1} = N_{v^i} \quad \text{and} \quad y^{2i} = S_{v^i}.
\end{align*}
In other words, Theorem~\ref{thm:main} holds, once we note that $[M(w_1,w_2)^{-1}]_{c_1,c_2} = g(w_1,w_2,c_1,c_2)$ as defined in Equations \eqref{eq:gdef1} to \eqref{eq:gdef3}.

\end{thm}

\begin{proof}

Immediately, let us use Equation~\ref{eq:defU} to write that
\begin{equation*}
    \mathbb{P} \left( \bigcap_{i=1}^k V_{t_i}(v^i) \right) = P\left( \bigcap_{i=1}^k U_{t_i}(v^i) \right).
\end{equation*}

Let us write
\begin{align*}
U_2(v) = U_2^0(v) \cup U_2^1(v),
\end{align*}
where 
\begin{align*}
U_2^0(v) &:= \{ \orho (W_v) = N_v \} \cap \{ \orho (S_v) = E_v \}, \\
U_2^1(v) &:= \{ \orho (W_v) = E_v \} \cap \{ \orho (S_v) = N_v \},
\end{align*}
are both events for the generalised snake configuration $\orho$. Now define
\begin{align*}
T := \{ 1 \leq j \leq k : t_j = 2 \}.
\end{align*}
By disjointness, we may then write
\begin{align*}
P\left(\bigcap_{i=1}^k U_{t_i}(v^i)\right)
= \sum_{\phi\colon T\to\{0,1\}}P\left(\bigcap_{i\in T} U^{\phi(i)}_{2}(v^i)\;\cap\; \bigcap_{i\notin T} U_{t_i}(v^i)\right),
\end{align*}
which is a sum over all the $2^{\# T}$ functions from $\phi:T \to \{0,1\}$.  

Fix a function $\phi \colon T \to \{0,1\}$. Let us use the shorthand $U_\phi = \bigcap_{i\in T} U^{\phi(i)}_{2}(v^i)\;\cap\; \bigcap_{i\notin T} U_{t_i}(v^i)$. This intersection is precisely the set such that, for any $\orho \in U_\phi$, for each $i = 1,\dots,k$, the values of $\orho(S_{v^i})$ and $\orho(W_{v^i})$ are determined. Explicitly, if we fix an arbitrary $\orho_0 \in U_\phi$, and define
\begin{equation*}
x^{2i-1} = W_{v^i}, \quad  x^{2i} = S_{v^i}, \quad y^{2i-1}_\phi = \orho_0 (W_{v^i}), \quad  y^{2i}_\phi = \orho_0 (S_{v^i}),
\end{equation*}
then the set $U_\phi$ is precisely given by,
\begin{equation*}
    U_\phi = \bigcap_{i=1}^{2k}\{\orho (x^i) = y^i_\phi\}.
\end{equation*}

We emphasise that the function $\phi:T \to \{0,1\}$ affects the choice of $y^{2i-1}_\phi$ and $y^{2i}_\phi$ whenever $i \in T$. Indeed, for $i \in T$ we have
\begin{align*}
\phi(i) = 0 &\implies y^{2i-1}_\phi= N_{v^i}, \quad  y^{2i}_\phi = E_{v^i}\\
\phi(i) = 1 &\implies y^{2i-1} = E_{v^i}, \quad  y^{2i} = N_{v^i}.
\end{align*}
Note that in the case that $\phi(i) = 0$ for all $i \in T$, the variables $\{x^1,\dots,x^{2k},y^1_\phi,\dots,y^{2k}_\phi\}$ coincide with those defined in the statement of the theorem. 

We are now in a position to apply Theorem~\ref{thm:localcorrGS} with $2k$ in place of $k$. Given our points $y^1_\phi, \dots, y^{2k}_\phi$ associated with our function $\phi:T \to \{0,1\}$, we have 
\begin{align}\label{eq:P(Uphi)}
P \left( U_\phi \right) &= P\left(\bigcap_{i=1}^{2k}\{\orho (x^i) = y^i_\phi\}\right) \nonumber\\
&= \prod_{i=1}^{2k} K_{00}(x^i,y^i) \det_{i,j=1}^{2k} L(x^i,y^j_\phi). 
\end{align} 
But now, we can see that 
\begin{align*}
\prod_{i=1}^{2k} K_{00}(x^i,y^i_\phi) = (\gamma^2)^{ \# \{ j \in T : \phi(j) = 0 \} } (\alpha \beta)^{ \# \{ j \in T : \phi(j) = 1 \} } \prod_{j \notin T} a_{t_j}.
\end{align*}
Moreover, swapping the value of $\phi(i)$ from $0$ to $1$ simply swaps two columns of the determinant in Equation~\eqref{eq:P(Uphi)}. In other words,
\begin{align*}
P \left( U_\phi \right) = (-\alpha\beta/\gamma^2)^{ \# \{ j \in T : \phi(j) = 1 \} }  P \left( U_{\phi_0} \right).
\end{align*} 
where $\phi_0(i) = 0$ for all $i \in T$. 

Hence, summing over all functions $\phi$, we have 
\begin{align*}
\sum_{ \phi :T \to \{0,1\} } P \left( U_\phi \right) &= \left(1-\frac{\alpha \beta}{\gamma^2} \right)^{\# T}  P \left( U_{\phi_0} \right)\\
&= \left(1-\frac{\alpha \beta}{\gamma^2} \right)^{\# T} \prod_{j \in T} (\gamma^2) \prod_{j \notin T} a_{t_j} \det_{i,j=1}^{2k} L(x^i,y^j). 
\end{align*} 
The result now follows from the fact that $a_2 = \gamma^2 - \alpha \beta$. 
\end{proof} 

Let us conclude by considering the case $k=1$, which describes the probabilities associated with a single vertex. In other words, we will prove Theorem~\ref{thm:frequency}.

\begin{proof}[Proof of Theorem~\ref{thm:frequency}]

By translation invariance, we may assume without loss of generality that our vertex $v$ lies at $\mathbf{0} = (0,0)$ in $\tor$. Now by setting $k=1$ in Theorem \ref{thm:premain} we have 
\begin{align*}
\mathbb{P}( V_{t}(\mathbf{0}) ) = a_t \det_{i,j=1}^2 L(x^i,y^j),
\end{align*}
where
\begin{align*}
x^1 = W_\mathbf{0} = (-1/2,0) \quad \text{and} \quad x^2 = S_\mathbf{0} = (0,-1/2),
\end{align*}
and 
\begin{align*}
\tau = 1 &\implies y^{1} = W_{\mathbf{0}} \quad \text{and} \quad y^{2} = S_{\mathbf{0}}\\
\tau = 2 &\implies y^{1} = N_{\mathbf{0}} \quad \text{and} \quad y^{2} = E_{\mathbf{0}}\\
\tau = 3 &\implies y^{1} = W_{\mathbf{0}} \quad \text{and} \quad y^{2} = N_{\mathbf{0}}\\
\tau = 4 &\implies y^{1} = E_{\mathbf{0}} \quad \text{and} \quad y^{2} = S_{\mathbf{0}}\\
\tau = 5 &\implies y^{1} = W_{\mathbf{0}} \quad \text{and} \quad y^{2} = E_{\mathbf{0}}\\
\tau = 6 &\implies y^{1} = N_{\mathbf{0}} \quad \text{and} \quad y^{2} = S_{\mathbf{0}}.
\end{align*}
Recall the notation,
\begin{align*}
\mathrm{d}\mathbf{w} := \frac{ \mathrm{d}w_1}{2 \pi \iota w_1} \frac{ \mathrm{d}w_2}{2 \pi \iota w_2} \frac{ \mathrm{d}\tilde{w}_1}{2 \pi \iota\tilde{w}_1} \frac{ \mathrm{d}\tilde{w}_2}{2 \pi \iota \tilde{w}_2}. 
\end{align*}
Now, changing the order of integration and taking determinants, we have
\begin{align*}
a_t\det_{i,j=1}^{2}L(x^i,y^j) &= a_t\det_{i,j=1}^{2} \left( \int_{(S^1)^2} \frac{dw_1}{2\pi i w_1}\,\frac{dw_2}{2\pi i w_2}\,w_1^{y^j_1 -x^i_1 }w_2^{y^j_2 - x^i_2}\big[M(w_1,w_2)^{-1}\big]_{c(x^i),\,c(y^j)}\right) \\[2pt]
&= \int_{(S^1)^{4}} 
\scalebox{0.9}{$
a_t\det \begin{bmatrix}
w_1^{y^1_1-x^1_1}w_2^{y^1_2-x^1_2}M(w_1,w_2)^{-1}_{c(x^1),c(y^1)} &  w_1^{y^2_1-x^1_1}w_2^{y^2_2-x^1_2}M(w_1,w_2)^{-1}_{c(x^1),c(y^2)} \\
\tilde{w}_1^{y^1_1-x^2_1}\tilde{w}_2^{y^1_2-x^2_2}M(\tilde{w}_1,\tilde{w}_2)^{-1}_{c(x^2),c(y^1)} & \tilde{w}_1^{y^2_1-x^2_1}\tilde{w}_2^{y^2_2-x^2_2}M(\tilde{w}_1,\tilde{w}_2)^{-1}_{c(x^2),c(y^2)}
\end{bmatrix}
$}
\mathrm{d}\mathbf{w}\\[2pt]
 &=: \int_{(S^1)^4}  \frac{1}{\Delta \tilde{\Delta}}	a_t \det R_t(\mathbf{w}) \mathrm{d}\mathbf{w},
\end{align*}
where we are using the obvious shorthand $\Delta := \Delta(w_1,w_2)$ and $\tilde{\Delta} := \Delta(\tilde{w}_1,\tilde{w}_2)$. There are six matrices $\{R_t(\mathbf{w}): 1 \leq t \leq 6 \}$ whose determinants we now calculate using the definitions of $x^1,x^2,y^1,y^2$, and the definition of $M(w_1,w_2)^{-1}$, which we recall from Equation~\eqref{eq:defM}. We may also define \begin{equation*}
    f_t(\mathbf{w}) = a_t\det R_t(\mathbf{w}).
\end{equation*} 

\begin{table}[!h]
\vspace{6pt}
\begin{center}
\begin{tabular}{c@{\hspace{20pt}}c@{\hspace{20pt}}c}
    $t$ & $R_t(\mathbf{w})$ & $f_t(\mathbf{w})$\\[12pt]
    $1$ & 
    $\begin{bmatrix}
    1+\beta w_2  & - \gamma w_1 \\
    - \gamma \tilde{w}_2 & 1 + \alpha \tilde{w}_1
    \end{bmatrix}$ 
    & $(1 + \beta w_2)(1 + \alpha \tilde{w}_1) - \gamma^2 w_1 \tilde{w}_2$\\[12pt]
    $2$ & 
    $\begin{bmatrix}
    - \gamma w_1 w_2 & w_1 ( 1 + \beta w_2)  \\
    \tilde{w}_2 ( 1 + \alpha \tilde{w}_1) & - \gamma \tilde{w}_1 \tilde{w}_2
    \end{bmatrix}$
    & $(\gamma^2 - \alpha \beta)\left(\gamma^2 \tilde{w}_1 w_2 - (1+\alpha \tilde{w}_1)(1+\beta w_2)\right)w_1 \tilde{w}_2$\\[12pt]
    $3$ & 
    $\begin{bmatrix}
    1 + \beta w_2 & -\gamma w_1 w_2  \\
    -\gamma \tilde{w}_2  & \tilde{w}_2(1 + \alpha \tilde{w}_1)
    \end{bmatrix}$ 
    & $\beta\tilde{w}_2\left((1+\beta w_2)(1+\alpha \tilde{w}_1) - \gamma^2 w_1 w_2\right)$\\[12pt]
    $4$ & 
    $\begin{bmatrix}
    (1 + \beta w_2)w_1 & -\gamma w_1  \\
    -\gamma \tilde{w}_1 \tilde{w}_2  & 1 + \alpha \tilde{w}_1
    \end{bmatrix}$ 
    & $\alpha w_1\left((1+\beta w_2)(1+\alpha \tilde{w}_1) - \gamma^2 \tilde{w}_1 \tilde{w}_2\right)$ \\[12pt]
    $5$ &
    $\begin{bmatrix}
    1 + \beta w_2 & (1 + \beta w_2)w_1 \\
    -\gamma \tilde{w}_2 & -\gamma \tilde{w}_1 \tilde{w}_2
    \end{bmatrix}$ 
    & $\gamma^2 \tilde{w}_2(1+\beta w_2)(w_1 - \tilde{w}_1)$\\[12pt]
    $6$ &$\begin{bmatrix}
    -\gamma w_1 w_2 & -\gamma w_1 \\
    (1+\alpha \tilde{w}_1)\tilde{w}_2 & 1+\alpha \tilde{w}_1
    \end{bmatrix}$
    & $\gamma^2 w_1(1+\alpha \tilde{w}_1)(\tilde{w}_2 - w_2)$ \\
    \vspace{0.1pt}
\end{tabular}
\end{center}
\caption{The matrices $R_t(\mathbf{w})$ and vertex polynomials $f_t(\mathbf{w})$ for each vertex type $t$.}
\label{tab:vertpolys}
\end{table}

We recall from the introduction (in particular, Equations \eqref{eq:vertpoly1} to \eqref{eq:vertpoly6}) that we called $\{f_t(\mathbf{w}):1 \leq t \leq 6\}$ the vertex polynomials. It is a simple calculation, which we will demonstrate shortly, to ensure that our definitions here coincide with the previously claimed forms.

When $t = 1$, we have $y^1 = x^1 = W_\mathbf{0}$, which are black, and $y^2 = x^2 = S_\mathbf{0}$, which are white. In particular,
\begin{align*}
R_1(\mathbf{w}) &=  \begin{bmatrix}
1+\beta w_2  & - \gamma w_1^{1/2}w_2^{1/2} (  w_1^{1/2} w_2^{-1/2} ) \\
- \gamma \tilde{w}_1^{1/2}\tilde{w}_2^{1/2} ( \tilde{w}_1^{-1/2} \tilde{w}_2^{1/2} )& 1 + \alpha \tilde{w}_1
\end{bmatrix}
= \begin{bmatrix}
1+\beta w_2  & - \gamma w_1 \\
- \gamma \tilde{w}_2 & 1 + \alpha \tilde{w}_1
\end{bmatrix},
\end{align*}
so
\begin{equation*}
    \mathbb{P}(V_1(\mathbf{0})) =\int_{(S^1)^4}\frac{1}{\Delta \tilde{\Delta}} f_1(\mathbf{w}) \mathrm{d}\mathbf{w}
\end{equation*}
where, because $a_1 = 1$,
\begin{equation*}
    f_1(\mathbf{w}) = (1 + \beta w_2)(1 + \alpha \tilde{w}_1) - \gamma^2 w_1 \tilde{w}_2.
\end{equation*}
Similar calculations for $t = 2$ through $6$ yield the values in Table~\ref{tab:vertpolys}, which all agree with the vertex polynomials as claimed. 

\end{proof}

One can and should check that $\sum_{t = 1}^6 f_t(\mathbf{w}) = \Delta \tilde{\Delta}$. This implies that
\begin{align*}
    \sum_{t=1}^6 \mathbb{P}(V_t(\mathbf{0})) = \sum_{t=1}^6\int_{(S^1)^4} \frac{1}{\Delta \tilde{\Delta}} f_t(\mathbf{w}) \ \mathrm{d}\mathbf{w}
    = \int_{(S^1)^4} \ \mathrm{d}\mathbf{w}
    &= 1,
\end{align*}
as expected.

\subsection{Application in Correlation Asymptotics} \label{sec:corr}

In this section we study the asymptotic behaviour of correlations between events in random six-vertex configurations on $\mathbb{Z}^2$. Recall that for a type $t \in \{1,\ldots,6\}$ and a vertex $(x,y) \in \mathbb{Z}^2$, $V_t(x,y)$ denotes the event that a six-vertex configuration has a vertex of type $t$ at $(x,y)$. We are interested in the asymptotic decay of correlations between vertex events at $(0,0)$ and at $(x,y)$, a large distance $R = \sqrt{x^2+y^2}$ away from the origin. As mentioned in the introduction, it is well known that these correlations decay quadratically, however, we now demonstrate using our main result that it is possible to obtain fine estimates for precisely how this decay occurs, showing in fact that the coefficient of the quadratic decay has sinusoidal behaviour. Indeed, we now use our main result, 
Theorem \ref{thm:main}, to study the fine behaviour of the large-$R$ asymptotics of quantities
\begin{align*}
C_{t_1,t_2}(x,y):=     \mathbb{P}^{\alpha,\beta,\gamma}(V_{t_1}(\mathbf{0}),V_{t_2}((x,y))) -     \mathbb{P}^{\alpha,\beta,\gamma}(V_{t_1}(\mathbf{0})) \mathbb{P}^{\alpha,\beta,\gamma} (V_{t_2}((x,y))), 
\end{align*}
There are, however, many different combinations of $t_1,t_2$ to choose from. Thus for the sake of simplicity we will focus on the case $t_1 = t_2 = 1$, and write
\begin{align*}
C(x,y) := C_{1,1}(x,y) 
\end{align*}
for short. Thus, in the remainder of this section we focus on the asymptotic behaviour of 
\begin{equation*}
    C(x,y) \coloneqq C_{1,1}(x,y) = \mathbb{P}^{\alpha,\beta,\gamma}(V_{1}(\mathbf{0}),V_{1}(x,y)) - \mathbb{P}^{\alpha,\beta,\gamma}(V_{1}(\mathbf{0}))^2,
\end{equation*}
bearing in mind that different regimes for the weighting parameters $\alpha,\beta,\gamma$ will have a substantial effect on the behaviour of this decay.

Let us write that $x^1 = (-1/2,0)$, $x^2 = (0,-1/2)$, $x^3 = (x-1/2,y)$, $x^4 = (x,y-1/2)$, and $y^i = x^i$ for all $1 \leq i \leq 4$. From Theorem~\ref{thm:premain}, we have that

\begin{align*}
    C(x,y) + \mathbb{P}^{\alpha,\beta,\gamma}(V_{1}(\mathbf{0}))^2 &= \det_{i,j = 1}^4 L(x^i,y^i) 
    = \det \begin{bmatrix}
        M(0,0) & M(x,y) \\
        M(-x,-y) & M(0,0)
    \end{bmatrix}
\end{align*}
where $M(x,y)$ is a $2 \times 2$ block matrix, given by
\begin{equation}\label{eq:Mdef}
    M(x,y) = \begin{bmatrix}
        r(x,y) + \beta r(x,y+1) & -\gamma r(x+1,y)\\
        -\gamma r(x,y+1) & r(x,y) + \alpha r(x+1,y)
    \end{bmatrix},
\end{equation}
with
\begin{equation}\label{eq:rdef}
    r(x,y) \coloneqq \oint \oint  \frac{dw_1}{2\pi \iota w_1} \frac{dw_2}{2\pi \iota w_2} \frac{w_1^x w_2^y }{\Delta(w_1,w_2)},
\end{equation}
where we recall that
\begin{align*}
\Delta(w_1,w_2) = (1 + \alpha w_1) (1 + \beta w_2) - \gamma^2 w_1 w_2 = 1 + \alpha w_1 + \beta w_2 - \Gamma w_1w_2,
\end{align*}
with $\Gamma = \gamma^2 - \alpha \beta$. Therefore, we are required to calculate the asymptotics of $r(x,y)$. 

The behaviour of the asymptotics of $r$ largely depends on the parameter regime we lie in. The particular condition of interest is the following. Let us say that $(\alpha,\beta,\gamma)$ lie in the \textbf{disordered regime} if we have all of
\begin{align*}
1+\alpha + \beta &> \Gamma,\\
1+\alpha + \Gamma &> \beta,\\
1 + \beta + \Gamma &> \alpha,\\
\alpha + \beta + \Gamma &> 1.
\end{align*}
Geometrically, this is equivalent to $1, \alpha, \beta$ and $\Gamma$ being able to form the sides of a \textbf{cyclic quadrilateral}. The cyclic quadrilateral that is formed by these edge lengths is unique up to rotation and reflection, and several constants that will appear in this regime can be realised as geometric quantities in this cyclic quadrilateral.

First, let us demonstrate the behaviour in the complementary open regimes. In any such regime, one parameter must dominate the others: for $\omega \in \{1,\alpha,\beta,\Gamma\}$, we say that $\omega$ \textbf{dominates} if $2\omega > 1 + \alpha + \beta + \Gamma$. We call these the \textbf{frozen regimes}.

\begin{proposition}\label{prop:frozenregime}
    In any frozen regime, the model collapses trivially. In particular:
    \begin{itemize}
        \item if $1$ dominates, $\mathbb{P}^{\alpha,\beta,\gamma}(V_1(\mathbf{0})) = 1$ (every vertex has type $1$ a.s.);
        \item if $\Gamma$ dominates, $\mathbb{P}^{\alpha,\beta,\gamma}(V_2(\mathbf{0})) = 1$ (every vertex has type $2$ a.s.);
        \item if $\beta$ dominates, $\mathbb{P}^{\alpha,\beta,\gamma}(V_3(\mathbf{0})) = 1$ (every vertex has type $3$ a.s.);
        \item if $\alpha$ dominates, $\mathbb{P}^{\alpha,\beta,\gamma}(V_4(\mathbf{0})) = 1$ (every vertex has type $4$ a.s.).     
    \end{itemize}
\end{proposition}
\begin{proof}
    These statements are simple to verify from Theorem~\ref{thm:frequency}. For example, suppose $1$ dominates. Then there exists $\epsilon > 0$ such that $\Delta(w_1,w_2)^{-1}\Delta(\tilde{w}_1,\tilde{w}_2)^{-1}$ is holomorphic on $\{w:|w|< 1+\epsilon\}^4$. Since $f_1(\mathbf{0})=1$ and $\Delta(0,0)=1$, the contour integral extracts the constant coefficient of a holomorphic function, and hence
    \begin{align*}
        \mathbb{P}^{\alpha,\beta,\gamma}(V_1(\mathbf{0})) = \int_{(S^1)^4} 
        \frac{ \mathrm{d}w_1}{2 \pi \iota w_1} \frac{ \mathrm{d}w_2}{2 \pi \iota w_2} \frac{ \mathrm{d}\tilde{w}_1}{2 \pi \iota \tilde{w}_1} \frac{ \mathrm{d}\tilde{w}_2}{2 \pi \iota \tilde{w}_2} \frac{ f_1(\mathbf{w})}{ \Delta(w_1,w_2) \Delta(\tilde{w}_1,\tilde{w}_2) }
        = 1.
    \end{align*}
    
The other cases follow by analogous Laurent expansions. For example, if $\Gamma$ dominates, then factoring out the dominant term $-\Gamma w_1 w_2$ gives a convergent geometric expansion on the unit torus, since $\Gamma > 1+\alpha+\beta$, from which $\mathbb{P}^{\alpha,\beta,\gamma}(V_2(\mathbf{0})) = 1$ follows immediately. 
\end{proof}

Now we turn our attention to the disordered regime. We first recall the following consequence of \cite[Lemma~4.4]{KOS}, which describes the Fourier coefficients of the reciprocal of a function with simple zeroes on the unit torus.

\begin{lemma}\label{lem:KOSfourier}
Suppose that $R$ is smooth on the unit torus and has a simple zero $(z_0,w_0)$ with
\[
R(z,w)=a(z-z_0)+b(w-w_0)+O\left(|z-z_0|^2+|w-w_0|^2\right).
\]
Then the contribution of this zero to
\[
\frac{1}{(2\pi\iota)^2}
\int_{\mathbb{T}^2}
\frac{z^xw^y}{R(z,w)}
\frac{dz}{z}\frac{dw}{w}
\]
is
\[
\mp\frac{z_0^xw_0^y}
{2\pi\iota(yaz_0-xbw_0)}
+
O\left(\frac{1}{x^2+y^2}\right),
\]
where the sign is $-$ if $\mathrm{Im}(-bw_0/(az_0))>0$ and $+$ if $\mathrm{Im}(-bw_0/(az_0))<0$. If there are several simple zeroes, their contributions are added.
\end{lemma}

To apply this result to
\[
\Delta(z,w)=1+\alpha z+\beta w-\Gamma zw,
\]
observe that a zero $(z_0,w_0)$ lies on the unit torus precisely when, writing $w_0=e^{\iota\theta_0}$,
\[
\left|
\frac{\Gamma w_0-\alpha}{1+\beta w_0}
\right|
=1,
\qquad
z_0=\frac{1+\beta w_0}{\Gamma w_0-\alpha}.
\]
Squaring the first equality gives
\[
\cos\theta_0
=
\frac{\Gamma^2+\alpha^2-1-\beta^2}
{2(\beta+\alpha\Gamma)}.
\]

\begin{figure}[!htb] \centering \includegraphics[width=0.35\textwidth]{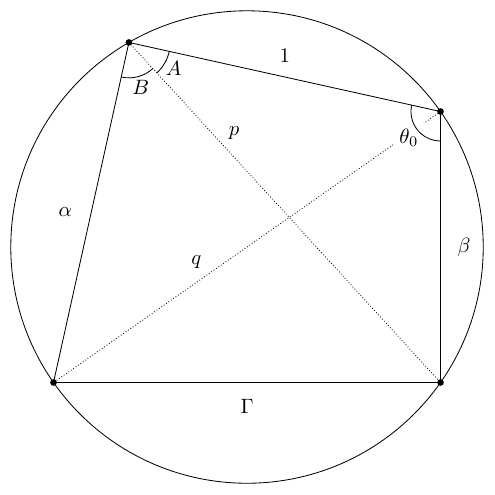} \caption{Cyclic quadrilateral showing the geometric interpretations of the new constants.} \label{fig:cyclicquad} \end{figure}

In the disordered regime the right-hand side lies in $(-1,1)$, so there is a unique $\theta_0\in(0,\pi)$. Define
\begin{align}\label{eq:kangaroo2}
pe^{\iota A}
=
1+\beta e^{\iota\theta_0},
\qquad
pe^{\iota B}
=
\Gamma e^{\iota\theta_0}-\alpha,
\qquad
q=\frac{\gamma^2}{p}.
\end{align}
Then the two zeroes of $\Delta$ on the unit torus are
\[
(z_0,w_0)
=
\left(e^{\iota(A-B)},e^{\iota\theta_0}\right)
\quad\text{and}\quad
(\overline{z_0},\overline{w_0}).
\]

\begin{lemma}\label{lem:rasymp}
Set
\[
\mathbf{j}\coloneqq(A-B,\theta_0)\in\mathbb{R}^2,
\qquad
\mathbf{k}\coloneqq
\pi\left(
-qe^{-\iota(\theta_0-B)},
\,pe^{-\iota A}
\right)\in\mathbb{C}^2.
\]
Here and below, $\langle\cdot,\cdot\rangle$ denotes the standard Hermitian inner product on $\mathbb{C}^2$, linear in its first argument. Then, as $|\mathbf{x}|\to\infty$ with $\mathbf{x}=(x,y)\in\mathbb{Z}^2$,
\[
r(\mathbf{x})
=
\mathrm{Im}\left(
e^{\iota\langle\mathbf{j},\mathbf{x}\rangle}
\frac{\langle\mathbf{k},\mathbf{x}\rangle}
{|\langle\mathbf{k},\mathbf{x}\rangle|^2}
\right)
+
O(|\mathbf{x}|^{-2}).
\]
\end{lemma}

\begin{proof}
At $(z_0,w_0)$ we have
\[
a=\partial_z\Delta(z_0,w_0)
=
\alpha-\Gamma w_0
=
-pe^{\iota B},
\]
while, using $z_0=(1+\beta w_0)/(\Gamma w_0-\alpha)$ and
$\Gamma+\alpha\beta=\gamma^2$,
\[
b=\partial_w\Delta(z_0,w_0)
=
\beta-\Gamma z_0
=
-\frac{\gamma^2}{\Gamma w_0-\alpha}
=
-qe^{-\iota B}.
\]
We now look to determine the sign of the imaginary part of  
\begin{align*}
 -\frac{bw_0}{az_0} = -\frac{q}{p}e^{-\iota(A+B-\theta_0)}.
\end{align*}
Note from $pe^{\iota A}=1+\beta e^{\iota\theta_0}$ and $pe^{\iota B}=\Gamma e^{\iota\theta_0}-\alpha$  that 
\begin{align*}
 e^{\iota(A+B-\theta_0)} = \frac{(1+\beta e^{\iota\theta_0}) (\Gamma e^{\iota\theta_0}-\alpha)e^{-\iota\theta_0}}{p^2}, 
 \end{align*}
 which has positive imaginary part $(\alpha+\beta \Gamma) \sin \theta_0/ p^2$. Therefore, $\mathrm{Im}\left(-\frac{bw_0}{az_0}\right)>0.$

Thus Lemma~\ref{lem:KOSfourier} gives the contribution
\[
-\frac{z_0^xw_0^y}{2\pi\iota(yaz_0-xbw_0)}
=
\frac{1}{2\iota}
e^{\iota\langle\mathbf{j},\mathbf{x}\rangle}
\frac{\langle\mathbf{k},\mathbf{x}\rangle}
{|\langle\mathbf{k},\mathbf{x}\rangle|^2}
\]
from $(z_0,w_0)$, where we have used
\[
yaz_0-xbw_0
=
-\frac{1}{\pi}\overline{\langle\mathbf{k},\mathbf{x}\rangle}.
\]
At the conjugate zero the sign of
$\mathrm{Im}(-bw_0/(az_0))$ is reversed, and its contribution is the
complex conjugate of the preceding expression. Adding the two contributions
therefore gives
\[
r(\mathbf{x})
=
\mathrm{Im}\left(
e^{\iota\langle\mathbf{j},\mathbf{x}\rangle}
\frac{\langle\mathbf{k},\mathbf{x}\rangle}
{|\langle\mathbf{k},\mathbf{x}\rangle|^2}
\right)
+
O(|\mathbf{x}|^{-2}),
\]
as required.
\end{proof}

For the correlation calculation we will also need three values of $r$ at the origin. These can be obtained directly from \eqref{eq:rdef}, without requiring a general one-dimensional contour representation.

\begin{lemma}\label{lem:rsmallvalues}
If $(\alpha,\beta,\gamma)$ lie in the disordered regime, then
\[
r(0,0)=\frac{\theta_0-A}{\pi},
\qquad
r(0,1)=\frac{A}{\pi\beta},
\qquad
r(1,0)=\frac{B-\theta_0}{\pi\alpha}.
\]
\end{lemma}

\begin{proof}
Let $C$ be the unit-circle arc from $e^{-\iota\theta_0}$ to $e^{\iota\theta_0}$ travelled anticlockwise, and let $C'$ be the complementary arc with the same endpoints travelled clockwise. Integrating first in $w_1$ in \eqref{eq:rdef}, the residue theorem gives
\[
r(0,y)
=
\int_C
\frac{dw}{2\pi\iota w}\,
\frac{w^y}{1+\beta w},
\qquad
r(1,y)
=
\int_{C'}
\frac{dw}{2\pi\iota w}\,
\frac{w^y}{\Gamma w-\alpha}.
\]
Indeed, the moving pole
\[
w_1=\frac{1+\beta w_2}{\Gamma w_2-\alpha}
\]
lies outside the unit circle precisely on $C$ and inside it on the complementary arc.

Using logarithms chosen continuously along the relevant arcs,
\begin{align*}
r(0,0)
&=
\frac{1}{2\pi\iota}
\left[
\log w-\log(1+\beta w)
\right]_{e^{-\iota\theta_0}}^{e^{\iota\theta_0}}
=
\frac{\theta_0-A}{\pi},\\
r(0,1)
&=
\frac{1}{2\pi\iota\beta}
\left[
\log(1+\beta w)
\right]_{e^{-\iota\theta_0}}^{e^{\iota\theta_0}}
=
\frac{A}{\pi\beta},
\end{align*}
while
\begin{align*}
r(1,0)
&=
\frac{1}{2\pi\iota\alpha}
\left[
\log(\Gamma w-\alpha)-\log w
\right]_{C'} = 
\frac{B-\theta_0}{\pi\alpha}. \qedhere
\end{align*}
\end{proof}
\begin{thm}\label{thm:corrasymp}
Correlations decay quadratically with an oscillatory leading coefficient. More precisely, let $\mathbf{x}=(x,y)$. There exist explicit constants $K_0\in\mathbb{R}$ and $K_1\in\mathbb{C}$, depending only on $\alpha,\beta,\gamma$, such that
\begin{equation}\label{eq:corr2} C(\mathbf{x}) = \frac{1}{2|\langle\mathbf{k},\mathbf{x}\rangle|^2} \left[ K_0\cos\!\left(2\langle\mathbf{j},\mathbf{x}\rangle\right) - \mathrm{Re}\left( K_1 \left( \frac{\langle\mathbf{k},\mathbf{x}\rangle} {|\langle\mathbf{k},\mathbf{x}\rangle|} \right)^2 \right) \right] + O(|\mathbf{x}|^{-3}). \end{equation}
\end{thm}

\begin{proof}

\begin{align}\label{eq:corr1}
C(x,y)
&=
\det
\begin{bmatrix}
M(\mathbf{0}) & M(x,y)\\
M(-x,-y) & M(\mathbf{0})
\end{bmatrix}
-
\mathbb{P}^{\alpha,\beta,\gamma}(V_1(\mathbf{0}))^2
\nonumber\\
&=
(\det M(\mathbf{0}))^2
\det\!\left(
I-M(-x,-y)M(\mathbf{0})^{-1}
M(x,y)M(\mathbf{0})^{-1}
\right)
-
\mathbb{P}^{\alpha,\beta,\gamma}(V_1(\mathbf{0}))^2,
\end{align}
where the last line follows from Schur's determinantal formula \cite{HJ}. By Theorem~\ref{thm:premain},
\begin{equation}\label{eq:probtype1}
\mathbb{P}^{\alpha,\beta,\gamma}(V_1(\mathbf{0}))
=
\det M(\mathbf{0}).
\end{equation}
Using $\det(I-A)=1-\tr(A)+\det(A)$ for a $2\times2$ matrix $A$, we obtain
\begin{align}\label{eq:C1}
C(x,y)
&=
-(\det M(\mathbf{0}))^2
\tr\!\left(
M(-x,-y)M(\mathbf{0})^{-1}
M(x,y)M(\mathbf{0})^{-1}
\right)
\nonumber\\
&\hspace{2cm}
+\det M(-x,-y)\det M(x,y)\\
&=
-(\det M(\mathbf{0}))^2
\tr\!\left(
M(-x,-y)M(\mathbf{0})^{-1}
M(x,y)M(\mathbf{0})^{-1}
\right)
+
O(|\mathbf{x}|^{-4}),
\nonumber
\end{align}
where Lemma~\ref{lem:rasymp} was used in the last line.

Applying Lemma~\ref{lem:rasymp} to the four entries in \eqref{eq:Mdef}, and observing that replacing $\langle\mathbf{k},\mathbf{x}\rangle$ by $\langle\mathbf{k},\mathbf{x}+(1,0)\rangle$ or $\langle\mathbf{k},\mathbf{x}+(0,1)\rangle$ changes the leading term by only $O(|\mathbf{x}|^{-2})$, gives
\begin{equation}\label{eq:Masymp}
M(x,y)
=
\frac{1}{|\langle\mathbf{k},\mathbf{x}\rangle|^2}
\mathrm{Im}\left(
\langle\mathbf{k},\mathbf{x}\rangle
e^{\iota\langle\mathbf{j},\mathbf{x}\rangle}
L
\right)
+
O(|\mathbf{x}|^{-2}),
\end{equation}
where $\mathrm{Im}$ is applied entrywise, and, using \eqref{eq:kangaroo2} we have 
\[
L
\coloneqq
\begin{bmatrix}
pe^{\iota A} & -\gamma e^{\iota(A-B)}\\
-\gamma e^{\iota\theta_0} & qe^{\iota(\theta_0-B)}
\end{bmatrix},
\]
and where to obtain the bottom right entry we have used the calculation
\begin{align*}
1 + \alpha e^{ i (A-B)} = 1 + \alpha z_0 = 1 + \alpha \frac{1+\beta w_0}{\Gamma w_0 - \alpha}= \frac{w_0 \gamma^2}{ \Gamma w_0 - \alpha} = q e^{ i (\theta_0 - B)}.
\end{align*}

By Lemma~\ref{lem:rsmallvalues},
\[
M(\mathbf{0})
=
\begin{bmatrix}
\frac{\theta_0}{\pi}
&
-\frac{\gamma(B-\theta_0)}{\pi\alpha}
\\[4pt]
-\frac{A\gamma}{\pi\beta}
&
\frac{B-A}{\pi}
\end{bmatrix}.
\]
Let
\[
J
\coloneqq
\operatorname{adj}(M(\mathbf{0}))
=
\begin{bmatrix}
\frac{B-A}{\pi}
&
\frac{\gamma(B-\theta_0)}{\pi\alpha}
\\[4pt]
\frac{A\gamma}{\pi\beta}
&
\frac{\theta_0}{\pi}
\end{bmatrix},
\qquad
T\coloneqq LJ.
\]
Since $(\det M(\mathbf{0}))M(\mathbf{0})^{-1}=J$, Equation~\eqref{eq:C1} becomes
\[
C(x,y)
=
-\tr\!\left(M(-x,-y)J M(x,y)J\right)
+
O(|\mathbf{x}|^{-4}).
\]
Using
\[
\langle\mathbf{k},-\mathbf{x}\rangle
=
-\langle\mathbf{k},\mathbf{x}\rangle,
\qquad
\langle\mathbf{j},-\mathbf{x}\rangle
=
-\langle\mathbf{j},\mathbf{x}\rangle,
\]
and substituting \eqref{eq:Masymp}, we obtain
\begin{align*}
C(x,y)
&=
\frac{1}{|\langle\mathbf{k},\mathbf{x}\rangle|^4}
\tr\left[
\mathrm{Im}\left(
\langle\mathbf{k},\mathbf{x}\rangle
e^{-\iota\langle\mathbf{j},\mathbf{x}\rangle}T
\right)
\mathrm{Im}\left(
\langle\mathbf{k},\mathbf{x}\rangle
e^{\iota\langle\mathbf{j},\mathbf{x}\rangle}T
\right)
\right]
+
O(|\mathbf{x}|^{-3}).
\end{align*}
Using
\[
\mathrm{Im}(A)\mathrm{Im}(B)
=
\frac14
\left(
A\overline{B}
+
\overline{A}B
-
AB
-
\overline{A}\,\overline{B}
\right)
\]
for matrices $A,B$, we obtain
\begin{align*}
C(x,y)
&=
\frac{1}{2|\langle\mathbf{k},\mathbf{x}\rangle|^4}
\left[
|\langle\mathbf{k},\mathbf{x}\rangle|^2
\tr(T\overline{T})
\cos\!\left(2\langle\mathbf{j},\mathbf{x}\rangle\right)
-
\mathrm{Re}\left(
\langle\mathbf{k},\mathbf{x}\rangle^2\tr(T^2)
\right)
\right]
+
O(|\mathbf{x}|^{-3}).
\end{align*}
The quantity $\tr(T\overline{T})$ is real. 
Thus \eqref{eq:corr2} follows on setting
\[
K_0\coloneqq\tr(T\overline{T}),
\qquad
K_1\coloneqq\tr(T^2).
\]
\end{proof}


\begin{thebibliography}{33}


\bibitem{ABPW}
\textsc{Aggarwal, A., Borodin, A., Petrov, L., \& Wheeler, M.} (2023).  
Free–Fermion Six Vertex Model: Symmetric Functions and Random Domino Tilings.  
\emph{Selecta Mathematica}. arXiv:2109.06718.

\bibitem{aggarwalshape}
\textsc{Aggarwal, A.} (2020).  
Limit shapes and local statistics for the stochastic six-vertex model.  
\emph{Communications in Mathematical Physics}, 376(1), 681–746.

\bibitem{AyyerChhitaJohansson2023}
\textsc{Ayyer, A., Chhita, S., \& Johansson, K.} (2023).
GOE fluctuations for the maximum of the top path in alternating sign matrices.
\emph{Duke Mathematical Journal} \textbf{172}(10), 1961–2104.

\bibitem{BelovReshetikhin20}
\textsc{Belov, P., \& Reshetikhin, N.} (2020).  
The two-point correlation function in the six-vertex model.  
arXiv:2012.05182.

\bibitem{grass1}
\textsc{Boos, H., Jimbo, M., Miwa, T., Smirnov, F., \& Takeyama, Y.} (2007).  
Hidden Grassmann structure in the XXZ model.  
\emph{Communications in Mathematical Physics}, 272(1), 263–281.

\bibitem{grass2}
\textsc{Boos, H., Jimbo, M., Miwa, T., Smirnov, F., \& Takeyama, Y.} (2009).  
Hidden Grassmann structure in the XXZ model II: Creation operators.  
\emph{Communications in Mathematical Physics}, 286(3), 875–932.


\bibitem{grass4}
\textsc{Boos, H., Jimbo, M., Miwa, T., \& Smirnov, F.} (2010).  
Hidden Grassmann structure in the XXZ model IV: CFT limit.  
\emph{Communications in Mathematical Physics}, 299(3), 825–866.

\bibitem{BCG}
\textsc{Borodin, A., Corwin, I., \& Gorin, V.} (2016).  
Stochastic six-vertex model.  
\emph{Duke Mathematical Journal}, 165(3), 563–624.

\bibitem{BG2}
\textsc{Borodin, A., \& Gorin, V.} (2019).  
A stochastic telegraph equation from the six-vertex model.  
\emph{Annals of Probability}, 47(6), 4137–4194.

\bibitem{BP}
\textsc{Borodin, A., \& Petrov, L.} (2018).  
Higher spin six-vertex model and symmetric rational functions.  
\emph{Selecta Mathematica}, 24(2), 751–874.

\bibitem{BPZ}
\textsc{Bogoliubov, N. M., Pronko, A. G., \& Zvonarev, M. B.} (2002).  
Boundary correlation functions of the six-vertex model.  
\emph{Journal of Physics A}, 35, 5525–5541.

\bibitem{baxter}
\textsc{Baxter, R. J.} (1982).  
\emph{Exactly Solved Models in Statistical Mechanics}.  
Academic Press.

\bibitem{bounew}
\textsc{Boutillier, C.} (2007). Pattern densities in non-frozen planar dimer models. Communications in mathematical physics, 271(1), 55-91.

\bibitem{BdT}
\textsc{Boutillier, C., \& De Tiliere, B.} (2014). Height representation of XOR-Ising loops via bipartite dimers.
\emph{Electronic Journal of Probability}, 19, no.\ 80, 1-33.

\bibitem{CJ}
\textsc{Chhita, S., \& Johansson, K.} (2016).  
Domino tilings of the Aztec diamond with two-periodic weights.  
\emph{Duke Mathematical Journal}, 165(14), 2587–2644.

\bibitem{CKP}
\textsc{Cohn}, H.
\& \textsc{Kenyon}, R.
\& \textsc{Propp}, J.
\newblock A variational principle for domino tilings.
\newblock {\em Journal of the American Mathematical Society} \textbf{14} (2001), 297--346.

\bibitem{ColomoPronko12}
\textsc{Colomo, F., \& Pronko, A. G.} (2012).  
An approach for calculating correlation functions in the six-vertex model with domain wall boundary conditions.  
\emph{Theoretical and Mathematical Physics}, 171, 641–654.

\bibitem{ColomoPronko21}
\textsc{Colomo, F., Di Giulio, G., \& Pronko, A. G.} (2021).  
Six-vertex model on a finite lattice: Integral representations for nonlocal correlation functions.  
arXiv:2107.13358.

\bibitem{Corwin2018}
\textsc{Corwin, I., Ghosal, P., Shen, H., \& Tsai, L.–C.} (2018).  
Stochastic PDE limit of the six-vertex model.  
\emph{Annals of Probability}, 46(5), 3176–3215.

\bibitem{dubedat}
\textsc{Dubédat, J.} (2011). Exact bosonization of the Ising model. arXiv preprint arXiv:1112.4399.


\bibitem{DC}
\textsc{Duminil-Copin, H., Kozlowski, K. K., Krachun, D., Manolescu, I., \& Tikhonovskaia, T.} (2022).  
On the six-vertex model’s free energy.  
\emph{Communications in Mathematical Physics}, 395, 1383–1430.

\bibitem{DCLQ}
\textsc{Duminil-Copin, H., Lis, M. \& Qian, W.} (2025).  
Conformal invariance of double random
currents I: Identification of the limit.  
\emph{Proceedings of the London Mathematical Society}, (3) 2025;130:e70022.

\bibitem{DCLQ2}
\textsc{Duminil-Copin, H., Lis, M. \& Qian, W.} (2021).  
Conformal invariance of double random
currents II: Tightness and properties in the discrete. arXiv preprint arXiv:2107.12880.

\bibitem{EKLP}
\textsc{Elkies, N., Kuperberg, G., Larsen, M. \& Propp, J.} 
Alternating Sign Matrices and Domino Tilings (Part I), 
\newblock {\em Journal of Algebraic Combinatorics}, {\bf 1}(2) (1992), 111–132.

\bibitem{falco}
\textsc{Falco, P.} (2013). Correlation critical exponents for the six-vertex model. arXiv preprint arXiv:1307.3002.


\bibitem{FerrariSpohn2006}
\textsc{Ferrari, P. L., \& Spohn, H.} (2006).  
Domino tilings and the six-vertex model at its free-fermion point.  
\emph{Journal of Physics A: Mathematical and General}, 39, 10297–10306.

\bibitem{GMT}
\textsc{Giuliani, A., Mastropietro, V., \& Toninelli, F. L.} (2017). Height fluctuations in interacting dimers. \emph{Annales de l’Institut Henri Poincaré, Probabilités et Statistiques} 53 (2017), 98–168.

\bibitem{GN}
\textsc{Gorin, V., \& Nicoletti, M.} (2025).  
Six-Vertex Model and Random Matrix Distributions.  
Bulletin of the American Mathematical Society 62(2), 175-234.

\bibitem{HJ}
\textsc{Horn, R. A., \& Johnson, C. R.} (2012).  
\emph{Matrix Analysis}. Cambridge University Press.

\bibitem{grass3}
\textsc{Jimbo, M., Miwa, T., \& Smirnov, F.} (2009).  
Hidden Grassmann structure in the XXZ model III: Introducing the Matsubara direction.  
\emph{Journal of Physics A}, 42(30), 304018.

\bibitem{grass5}
\textsc{Jimbo, M., Miwa, T., \& Smirnov, F.} (2011).  
Hidden Grassmann structure in the XXZ model V: Sine–Gordon model.  
\emph{Letters in Mathematical Physics}, 96(1–3), 325–365.

\bibitem{JS}
\textsc{Johnston, S. G. G., \& Shiatis, R.} (2025).  
The Integrable Snake Model,
\emph{Annals of Probability} (to appear), arXiv preprint arXiv:2501.15483.

\bibitem{kasteleyn1961}
\textsc{Kasteleyn, P. W.} (1961).  
The statistics of dimers on a lattice I: The number of dimer arrangements on a quadratic lattice.  
\emph{Physica}, 27(12), 1209–1225.

\bibitem{kasteleyn1967}
\textsc{Kasteleyn, P. W.} (1967).  
Graph theory and crystal physics.  
In F. Harary (Ed.), \emph{Graph Theory and Theoretical Physics}, 43–110. Academic Press.



\bibitem{kenyonaihp}
\textsc{Kenyon, R.} (1997).  
Local statistics of lattice dimers.  
\emph{Annales de l'Institut Henri Poincaré, Probabilités}, 33, 591–618.

\bibitem{KOS}
\textsc{Kenyon}, R.
\& \textsc{Okounkov}, A.
\& \textsc{Sheffield}, S.
\newblock Dimers and amoebae.
\newblock {\em Annals of Mathematics} \textbf{163} (2006), 1019--1056.


\bibitem{K}
\textsc{Korepin, V. E.} (1982).  
Calculation of norms of Bethe wave functions.  
\emph{Communications in Mathematical Physics}, 86, 391–418.

\bibitem{Korff2006}
\textsc{Korff, C.} (2006).  
The six-vertex model: a pedagogical overview.  
\emph{Kôkyûroku Bessatsu}, 1480, 113–134.

\bibitem{KuperbergASM}
\textsc{Kuperberg, G.} (1996).
Another proof of the alternating sign matrix conjecture.
\emph{International Mathematics Research Notices} \textbf{1996}(3), 139–150.

\bibitem{lis}
\textsc{Lis, M.} (2021).  
On delocalization in the six-vertex model.  
\emph{Communications in Mathematical Physics}, 383(2), 1181–1205.

\bibitem{Lieb1967}
\textsc{Lieb, E. H.} (1967).  
Exact solution of the problem of the entropy of two-dimensional ice.  
\emph{Physical Review}, 162(1), 162–172.

\bibitem{Lieb1967a}
\textsc{Lieb, E. H.} (1967).  
Exact solution of the F model of an antiferroelectric.  
\emph{Physical Review Letters}, 18, 692–694.

\bibitem{Lieb1967b}
\textsc{Lieb, E. H.} (1967).  
Residual entropy of square ice.  
\emph{Physical Review}, 162, 162–172.

\bibitem{M}
\textsc{Murasugi, K.} (1996).  
\emph{Knot Theory and its Applications}. Birkhäuser.

\bibitem{OR}
\textsc{Okounkov, A., \& Reshetikhin, N.} (2003).
Correlation Function of Schur Process with Application to Local Geometry of a Random 3-Dimensional Young Diagram.
\emph{Journal of the American Mathematical Society}, \textbf{16}(3), 581--603.

\bibitem{Pauling1935}
\textsc{Pauling, L.} (1935).  
The structure and entropy of ice and other crystals with some randomness of atomic arrangement.  
\emph{Journal of the American Chemical Society}, 57, 2680–2684.

\bibitem{Reshetikhin2017}
\textsc{Reshetikhin, N. and Sridhar, A.} (2017).  
Integrability of limit shapes of the six-vertex model.  
\emph{Communications in Mathematical Physics}. arXiv:1510.01053.

\bibitem{RobbinsRumseyASM}
\textsc{Robbins, D.\,P., \& Rumsey, H.} (1986).
Determinants and alternating sign matrices.
\emph{Advances in Mathematics} \textbf{62}, 169–184.

\bibitem{Sutherland1967}
\textsc{Sutherland, B.} (1967).  
Exact solution of a two-dimensional model for hydrogen-bonded crystals. 
\emph{Physical Review Letters}, 19, 103–104.

\bibitem{Wong}
\textsc{Wong, R.} (2001).
Asymptotic Approximation of Integrals.
\emph{Society for Industrial and Applied Mathematics}.

\bibitem{ZeilbergerASM}
\textsc{Zeilberger, D.} (1996).
Proof of the alternating sign matrix conjecture.
\emph{Electronic Journal of Combinatorics} \textbf{3}, R13.

\bibitem{ZinnJustin2009}
\textsc{Zinn–Justin, P.} (2009).  
Six-vertex, loop and tiling models: integrability and combinatorics.  
\emph{Journal of Physics A: Mathematical and Theoretical}, 42, 313001.

\end{thebibliography}
\end{document}